\documentclass[12pt]{amsart}
\usepackage{amsmath, amsthm, amscd, amssymb, amsfonts, latexsym}
\usepackage{fullpage}
\usepackage{bm}
\usepackage[all]{xy}
\usepackage{tikz}
\usepackage{fancybox}

\newcommand{\kommentar}[1]{}

\newcommand{\F}{\mathbb F}

\newcommand{\Z}{\mathbb Z}

\newcommand{\C}{\mathbb C}

\DeclareMathOperator{\re}{Re}

\renewcommand{\pmod}[1]{\,(\mathrm{mod}\,#1)}

\newtheorem{lem}{Lemma}[section]
\newtheorem{prop}[lem]{Proposition}
\newtheorem{thm}[lem]{Theorem}

\newtheorem{cor}[lem]{Corollary}
\newtheorem{conj}[lem]{Conjecture}
\theoremstyle{definition}

\newtheorem{rem}[lem]{Remark}

\newcommand{\mcom}[1]{{\color{orange}{Matilde: #1}} }

\author{Vivian Kuperberg}
\author{Matilde Lal\'in}

\address{Vivian Kuperberg: Stanford University,
Department of Mathematics,
450 Jane Stanford Way, Building 380,
Stanford, CA 94305-2125, USA }\email{viviank@stanford.edu}

\address{Matilde Lal\'in:  D\'epartement de math\'ematiques et de statistique,
                                    Universit\'e de Montr\'eal,
                                    CP 6128, succ. Centre-ville,
                                     Montreal, QC H3C 3J7, Canada}\email{mlalin@dms.umontreal.ca}

\thanks{This work is supported by NSF GRFP grant DGE-1656518,  the Natural Sciences and Engineering Research Council of Canada, Discovery Grant 355412-2013, the Fonds de recherche du Qu\'ebec - Nature et technologies, Projet de recherche en \'equipe 300951, and NSF FRG Grant 1854398 through the American Institute of Mathematics}

                                     
 \title{Sums of divisor functions and von Mangoldt convolutions in $\F_q[T]$ leading to symplectic distributions}

\begin{document}

\begin{abstract}
In \cite{KR3} Keating, Rodgers, Roditty-Gershon and Rudnick established relationships of the mean-square of sums of the divisor function $d_k(f)$ over short intervals and over arithmetic progressions for the function field $\mathbb{F}_q[T]$ to certain integrals over the ensemble of unitary matrices. We consider similar problems leading to distributions over the ensemble of symplectic matrices. We also consider analogous questions involving convolutions of the von
 Mangoldt function.    
\end{abstract}

\maketitle

\section{Introduction}

The goal of this paper is to study the connection between certain arithmetic sums in function fields and integrals over the ensembles of unitary and unitary symplectic matrices. We consider the $k$-th divisor function over $\F_q[T]$ and study two problems: the average over all the monic polynomials of fixed degree that yield a quadratic residue when viewed modulo a fixed monic irreducible polynomial $P$, and the average  over all the monic polynomials of fixed degree satisfying certain condition that is analogous to having an argument (in the sense of complex numbers) lying at certain specific sector of the unit circle. In both cases, we compute asymptotics for the average and the variance as $q \rightarrow \infty$ and we prove that the variance is described by a unitary symplectic matrix integral involving sum of products of secular coefficients. Our work is analogous to that of \cite{KR3}, but with problems that yield symplectic regimes, rather than unitary. We also consider the same setup with von Mangoldt convolutions instead of the $k$th divisor function and obtain similar results, connecting problems about the variance of sums of von Mangoldt convolutions to appropriate matrix integrals of sums of products of traces.

The \emph{$k$-th divisor function $d_k(n)$} over $\mathbb{Z}$ is the number of ways of writing a positive integer $n$ as a product of $k$ positive integers.  It provides the coefficients of the $k$th power of the Riemann zeta function:
\[\zeta(s)^k=\sum_{n=1}^\infty \frac{d_k(n)}{n^s}, \qquad \re(s)>1.\]
Now consider the remainder term of partial sums of the divisor function
\[\Delta_k(x):=\sum_{n\leq x} d_k(n) -\mathrm{Res}_{s=1}\frac{x^s\zeta(s)^k}{s}=\sum_{n\leq x} d_k(n) -x P_{k-1}(\log x),\]
where $P_{k-1}(T)$ is certain polynomial of degree $k-1$ (see \cite[Chapter XII]{Titchmarsh}). The mean square of $\Delta_k(x)$ has been computed by Cram\'er \cite{Cramer} for $k=2$ and by Tong \cite{Tong} for $k\geq 3$ (assuming the Riemann Hypothesis for $k\geq 4$) and was found to be 
\[\frac{1}{X}\int_{X}^{2X} \Delta_k(x)^2 dx \sim c_kX^{1-\frac{1}{k}},
\]
for certain constants $c_k$. Heath-Brown \cite{Heath-Brown-distribution} proved that $\Delta_k(x)/x^{\frac{1}{2}-\frac{1}{2k}}$ has a non-Gaussian limiting value distribution. 

Two particular problems have attracted attention: the distribution of the divisor function in short intervals and  the distribution of sums of the divisor function over arithmetic progressions.

The distribution in short intervals problem involves studying 
\[\Delta_k(x;H):=\Delta_k(x+H)-\Delta_k(x),\]
where $H<X^{1-1/k}$. 

In \cite[Conjecture 1.1]{KR3}, Keating, Rodgers, Roditty-Gershon, and Rudnick conjectured, based on their results over function fields, that if 
 $0<\delta <1-\frac{1}{k}$ is fixed, then for $H=X^\delta$, as $X\rightarrow \infty$, 
 \begin{equation}\label{eq:variancedelta}\frac{1}{X} \int_X^{2X} \Delta_k(x,H)^2 dx \sim a_k{\mathcal P}_k(\delta) H(\log H)^{k^2-1},
 \end{equation}
 where \begin{equation}\label{eq:ak}a_k=\prod_p \left( \Big(1-\frac{1}{p}\Big)^{k^2}\sum_{j=0}^\infty \Big(\frac{\Gamma(k+j)}{\Gamma(k)j!}\Big)^2\frac{1}{p^j}\right)
       \end{equation}
and ${\mathcal P}_k(\delta)$ is a piecewise polynomial function of $\delta$, of degree $k^2-1$, given by 
 \[{\mathcal P}_k(\delta)=(1-\delta)^{k^2-1} \gamma_k\left(\frac{1}{1-\delta}\right).\]
Here
 \begin{equation}\label{eq:gamma}
 \gamma_k(c)=\frac{1}{k!G(1+k)^2}\int_{[0,1]^k}\delta_c(w_1+\cdots +w_k) \prod_{i<j} (w_i-w_j)^2d^kw,
 \end{equation}
 where $\delta_c(w)=\delta(w-c)$ is the delta distribution translated by $c$, and $G$ is the Barnes $G$-function.

This conjecture is consistent with a Theorem of Lester \cite{Lester} for the interval $1-\frac{1}{k-1}<\delta <1-\frac{1}{k}$. Further evidence was found by Bettin and Conrey \cite{Bettin-Conrey} who proved that the Shifted moments conjecture for the Riemann zeta function implies that the main term of the variance of $\Delta_k(x,H)$ has coefficient $a_k\gamma_k(c)$, as predicted by \eqref{eq:variancedelta}.

The distribution over arithmetic progressions involves studying
\[\mathcal{S}_{d_k;X;Q}(A)=\sum_{\substack{n\leq X\\n\equiv A \pmod{Q}}}d_k(n).\]
The variance $\mathrm{Var}(\mathcal{S}_{d_k;X;Q})$ has been studied by Motohashi \cite{Motohashi}, Blomer \cite{Blomer}, and Lau and Zhao \cite{Lau-Zhao}, for $k=2$ and Kowalski and Ricotta \cite{Kowalski-Ricotta} for $k\geq 3$.

In \cite[Conjecture 3.3]{KR3} Keating, Rodgers, Roditty-Gershon, and Rudnick also conjectured, based on results that they obtained over function fields, that for $Q$ prime, $Q^{1+\varepsilon} <X<Q^{k-\varepsilon}$, as $X \rightarrow \infty$, 
\begin{equation}\label{eq:VarSdk}\mathrm{Var}(\mathcal{S}_{d_k;X;Q}) \sim \frac{X}{Q} a_k \gamma_k \left( \frac{\log X}{\log   Q}\right) (\log Q)^{k^2-1},\end{equation}
     where $a_k$ is given by \eqref{eq:ak} and $\gamma_k$ is given by \eqref{eq:gamma}.

Bettin and Conrey \cite{Bettin-Conrey} proved that the Shifted moments conjecture for Dirichlet $L$-functions implies that the main term of $\mathrm{Var}(\mathcal{S}_{d_k;X;Q})$ has coefficient $a_k\gamma_k(c)$, as predicted by \eqref{eq:VarSdk}. In sum, \cite{Bettin-Conrey} shows that a uniform version of the Shifted moments conjecture implies both Conjectures 1.1 and 3.3. from \cite{KR3}.

The same problems can be posed in the ring of polynomials $\F_q[T]$, where $q$ is an odd prime power. The benefit of working in the function field setting is that many of the analogous conjectures are provable. In $\F_q[T]$, $d_k(f)$ is defined for a monic polynomial $f$ to be 
\[d_k(f):=\# \{(f_1,\dots,f_k)\, : \, f=f_1\cdots f_k, \, f_j \mbox{ monic}\}.\] 
Let $U$ be an $N\times N$ matrix. The secular coefficients $\mathrm{Sc}_{j}(U)$ are the coefficients of the characteristic polynomial of $U$:
\[\det(1+xU)=\sum_{j=0}^N \mathrm{Sc}_{j}(U)x^j.\]

Keating, Rodgers, Roditty-Gershon, and Rudnick prove in Theorem 1.2 of \cite{KR3} that for $0 \le h \le \min\{n-5,(1-\frac{1}{k})n-2\}$, as $q \to \infty$,
\begin{align}\label{eq:KR3-1.2}
    \frac 1{q^n} \sum_{A \in \mathcal M_n} &\left|\sum_{f \in I(A;h)} d_k(f) - q^{h+1}\binom{n+k-1}{k-1}\right|^2\nonumber\\
    &=q^h \int_{\mathrm{U}(n-h-2)}\left| \sum_{\substack{j_1+\cdots+j_k=n\\0\leq j_1,\dots,j_k\leq n-h-2}}\mathrm{Sc}_{j_1}(U)\cdots \mathrm{Sc}_{j_k}(U)\right|^2 dU+O(q^{h-\frac{1}{2}}).
\end{align}

Here the integral is over the group $U(N)$ of $N\times N$ unitary matrices with respect to the Haar probability measure. They also showed in Theorem 3.1 of \cite{KR3} that for $Q$ squarefree and $n \le k(\deg(Q)-1)$,
\begin{align}\label{eq:KR3-3.1}
\lim_{q \to \infty} \frac{|Q|}{q^n} \frac 1{\Phi(Q)} \sum_{\substack{A \pmod Q \\ \gcd(A,Q) = 1}} &\left|\mathcal S_{d_k,n,Q}(A) - \langle \mathcal S_{d_k,n,Q} \rangle \right|^2 \nonumber \\
&= \int_{\mathrm{U}(\deg(Q)-1)}\left| \sum_{\substack{j_1+\cdots+j_k=n\\0\leq j_1,\dots,j_k\leq \deg(Q)-1}}\mathrm{Sc}_{j_1}(U)\cdots \mathrm{Sc}_{j_k}(U)\right|^2 dU.
\end{align}

The connection with Conjectures 1.1 and 3.3 is outlined in Theorem 1.5 of \cite{KR3}, which says that if $c = m/N$, for $c \in [0,k]$,
 \begin{equation}
 \int_{\mathrm{U}(N)}\left| \sum_{\substack{j_1+\cdots+j_k=m\\0\leq j_1,\dots,j_k\leq N}}\mathrm{Sc}_{j_1}(U)\cdots \mathrm{Sc}_{j_k}(U)\right|^2 dU=\gamma_k(c)N^{k^2-1}+O_k(N^{k^2-2}).
 \end{equation}
 Here again, the coefficient $\gamma_k(c)$ is given by equation \ref{eq:gamma}.

Both arithmetic questions above lead to distributions over the whole unitary group $U(N)$. Our aim is to present two similar questions which instead lead to distributions on the unitary symplectic group $\mathrm{Sp}(N)$.

The first question that we consider is the distribution of $d_k(f)$ when restricted to quadratic residues modulo an irreducible polynomial $P$. We use the notation that $\mathcal P_{2g+1}$ is the set of monic irreducible polynomials of degree $2g+1$.

\begin{thm}\label{thm:var}
 Let $n\leq 2gk$. Let
\[\mathcal{S}^S_{d_k,n}(P):=\sum_{\substack{f \,\text{monic}, \deg(f)=n  \\f\equiv \square \pmod{P}\\P\nmid f}} d_k(f),\]
where $P$ is a monic irreducible polynomial of degree $2g+1$.  As $q \rightarrow \infty$,
\[\mathcal{S}^S_{d_k,n}(P) \sim \frac{1}{2}\sum_{\substack{f\in \mathcal{M}_n \\P\nmid f}} d_k(f) \sim \frac{q^n}{2}\binom{k+n-1}{k-1},\]
and
 \begin{align*}
\frac{1}{\# \mathcal{P}_{2g+1}}\sum_{P\in \mathcal{P}_{2g+1}} \left|\mathcal{S}^S_{d_k,n}(P)- \frac{1}{2}\sum_{\substack{f\in \mathcal{M}_n \\P\nmid f}} d_k(f)\right|^2
 \sim & \frac{q^n}{4} \int_{\mathrm{Sp}(2g)}\left| \sum_{\substack{j_1+\cdots+j_k=n\\0\leq j_1,\dots,j_k\leq 2g}}\mathrm{Sc}_{j_1}(U)\cdots \mathrm{Sc}_{j_k}(U)\right|^2 dU.
\end{align*}
\end{thm}
The restriction to quadratic residues modulo an irreducible polynomial $P$ can be detected by twisting by a quadratic character. The presence of these quadratic twists ultimately yields the connection to a symplectic matrix integral. One could also consider the sum over twists by a quadratic character to a not-necessarily-irreducible modulus (as opposed to our sums, where the modulus must be an irreducible polynomial $P$). These sums yield similar results, although for simplicity we restrict our attention here to the sums with irreducible modulus.

One quirk of this choice is that its proof relies on the fact that the family $y^2=P(x)$, where $P(x)$ is a monic irreducible polynomial of fixed odd degree $2g+1$, has monodromy group $\mathrm{Sp}(2g)$. This is a more restrictive family than the somewhat more standard family $y^2 = f(x)$, for $f$ a monic squarefree polynomial of fixed odd degree $2g+1$. This is derived from the monodromy group of the hyperelliptic ensemble, by an argument that was provided to us by Katz; the monodromy arguments involved are presented in Section \ref{sec:monodromy}.

Analogously to the distribution over short intervals, we consider
 \[\mathcal{N}^S_{d_\ell,k,\nu}(v)=\sum_{\substack{f \,\text{monic}, \deg(f)=\nu\\ f(0)\not =0\\U(f)\in \mathrm{Sect}(v,k)}} d_\ell(f),\]
 where the sum is taken over monic polynomials of fixed degree with certain condition (see \eqref{eq:sector}) that can be interpreted as the function field analogue of having the argument of a complex number lying in certain specific sector of the unit circle. This follows the model of Gaussian integers in the function field context that was considered by Rudnick and Waxman in \cite{Rudnick-Waxman} and initially developped by Bary-Soroker, Smilansky, and Wolf in \cite{BSSW}.
 
 \begin{thm}\label{thm:sector}
 Let $\nu \leq \ell (2\kappa-2)$ with $\kappa=\left\lfloor \frac{k}{2}\right \rfloor$.  As $q\rightarrow \infty$, 
 \[\langle \mathcal{N}^S_{d_\ell,k,\nu}\rangle\sim q^{\nu-\kappa}\binom{\ell+\nu-1}{\ell-1},\]
  and
\[\frac{1}{q^\kappa}\sum_{u \in \mathbb{S}_k^1} \left|\mathcal{N}^S_{d_\ell,k,\nu}(u)-\langle \mathcal{N}^S_{d_\ell,k,\nu}\rangle\right|^2\sim\frac{q^\nu}{q^{\kappa}}\int_{\mathrm{Sp}(2\kappa-2)}\left| \sum_{\substack{j_1+\cdots+j_\ell=\nu\\0\leq j_1,\dots,j_\ell \leq 2\kappa-2}}\mathrm{Sc}_{j_1}(U)\cdots \mathrm{Sc}_{j_\ell}(U)\right|^2 dU.\]
\end{thm}
This result also rests on a monodromy theorem due to Katz \cite{Katz}.

We also study parallel results coming from the distribution of $k$-fold convolutions of the von Mangoldt function rather than from the distribution of the divisor function $d_k$.  The {\it  von Mangoldt function} $\Lambda(d)$ over $\Z$ is defined as $\Lambda(n)=\log p$ if $n=p^k$ for $p$ a prime number, and $0$ otherwise. It arises as the coefficients of the logarithmic derivative of the Riemann zeta function:
\[-\frac{\zeta'(s)}{\zeta(s)}=\sum_{n=1}^\infty \frac{\Lambda(n)}{n^s}.\]
The Prime Number Theorem says that 
\[\psi(x):=\sum_{n \leq x}\Lambda(n)\sim x \mbox{ as } x \rightarrow \infty.\]
The distribution of primes in short intervals involves studying 
\[\psi(x,H):=\sum_{x-\frac{H}{2}\leq n\leq x+\frac{H}{2}} \Lambda (n),\]
for $1\leq H \leq x$,  which due to the Prime Number Theorem has expected value $H$. The Riemann Hypothesis gives the estimate $\psi(x;H)\sim H$ for $H>x^{\frac{1}{2}+o(1)}$. Goldston and Montgomery \cite{GoldstonMontgomery} proved that for $X^\delta<H<X^{1-\delta}$, 
\begin{equation}\label{conj:var-interval}
\frac{1}{X}\int_2^X |\psi(x;H)-H|^2dx \sim H(\log X-\log H),
\end{equation}
under the Riemann Hypothesis and the strong pair correlation conjecture. 

Meanwhile the distribution of primes in arithmetic progressions involves studying 
\[\psi(X;Q,A):=\sum_{\substack{n\leq X\\n \equiv A \pmod{Q}}}\Lambda(n).\]
The Prime Number Theorem in arithmetic progressions states that  for $Q$ fixed,
\[\psi(X;Q,A) \sim \frac{X}{\phi(Q)}, \mbox{ as } X \rightarrow \infty,\] 
so the variance is given by
\[G(X,Q)=\sum_{\substack{A\pmod{Q}\\(A,Q)=1}}\left|\psi(X;Q,A)-\frac{X}{\phi(Q)} \right|^2.\]
The value of $G(X,Q)$ has been widely studied. Hooley \cite{Hooley} conjectured that, under certain unspecified conditions,  
\begin{equation}\label{conj:var-ap}
G(X,Q) \sim X \log Q.
\end{equation}
Friedlander and Goldston \cite{FriedlanderGoldston} conjectured that \eqref{conj:var-ap} holds for $X^{1/2+\varepsilon}<Q<X$ and proved it assuming a Hardy-Littlewood conjecture with small reminders. The barrier of $X\leq Q^{1/2}$ seems to be hard to overcome. 

Keating and Rudnick \cite{Keating-Rudnick} prove analogues of both \eqref{conj:var-interval} and \eqref{conj:var-ap} for function fields. Rudnick and Waxman \cite{Rudnick-Waxman} study analogues in the number field case, considering the variance of primes in sectors.

For a monic polynomial $f\in\F_q[T]$ monic, we define
\[\Lambda_k(f) := \sum_{\substack{f_1 \cdots f_k = f \\ f_i \text{ monic}}} \Lambda(f_1) \cdots \Lambda(f_k),\]
where $\Lambda(f)$ is the von Mangoldt function
\[\Lambda(f)=\begin{cases} \deg(P) & f=P^k,\\0 & \text{otherwise}, \end{cases}\]
and we let $\Lambda_k(cf):=\Lambda_k(f)$ for $f\in \F_q[T]$ monic and $c \in \F_q^\times$.

The following result in function fields corresponds to equation \eqref{eq:KR3-1.2}.
\begin{thm} \label{thm:N0-vM}
Let $A \in \F_q[T]$ be a monic polynomial of degree $n$, let $0 \le h \le n-2$, and define
\[\mathcal N^U_{0,\Lambda_k}(A;h) := 
\sum_{\substack{f \text{ monic}\\|f-A|\leq q^h\\ f(0) \ne 0}} \Lambda_k(f).\]
Assume $k\leq n$. Then as $q\rightarrow \infty$, 
\[\langle \mathcal{N}^U_{0,\Lambda_k}\rangle\sim q^{h+1}\binom{n-1}{k-1},\]
and 
 \[\frac{1}{q^n}\sum_{A \in \mathcal{M}_n}  \left|\mathcal{N}^U_{0,\Lambda_k}(A;h)- \langle \mathcal{N}^U_{0,\Lambda_k}\rangle\right|^2\sim q^{h+1}\int_{\mathrm{U}(n-h-2)}\left|\sum_{
 \substack{j_1+\cdots+j_k=n\\1\leq j_1,\dots,j_k }} 
 \mathrm{Tr}(U^{j_1})\cdots \mathrm{Tr}(U^{j_k}) \right|^2dU.\]
\end{thm}
In \cite[Theorem 1.3]{Rodgers-covariance}, Rodgers established a more general result about a slightly different function. He considered the covariance of higher-order von Mangoldt functions, defined recursively by 
\[\Lambda^0_k(f) :=\sum_{\substack{d\, \text{monic}\\d\mid f}}\Lambda^0_{k-1}(d)\Lambda(f/d)+\Lambda^0_{k-1}(f) \deg(f),\]
 and later developed a more general method to compute the variances of general arithmetic functions in short intervals \cite{Rodgers-arithmetic}.

Looking at the $k$-fold convolution in arithmetic progressions instead of short intervals, we get the following result, corresponding to equation \eqref{eq:KR3-3.1}.
\begin{thm} \label{thm:vM-U-arith} Let $Q\in \F_q[T]$ be square-free and $A\in \F_q[T]$ be coprime to $Q$, and define 
\[\mathcal S^U_{\Lambda_k, n,Q}(A):=\sum_{\substack{f \text{monic},\,\deg(f) = n \\ f \equiv A \pmod{Q}}} \Lambda_k(f).\]
Then, for $k\leq n$ and as $q\rightarrow \infty$,
\[\mathcal{S}^U_{\Lambda_k,n,Q}(A) \sim \frac{1}{\Phi(Q)}\sum_{\substack{f\in \mathcal{M}_n\\(f,Q)=1}}\Lambda_k(f)\sim \frac{q^n}{\Phi(Q)}\binom{n-1}{k-1},\]
and
\begin{align*}\frac{1}{\Phi(Q)}\sum_{\substack{A\pmod{Q}\\(A,Q)=1}}  &\left|\mathcal{S}^U_{\Lambda_k,n,Q}(A)- \frac{1}{\Phi(Q)}\sum_{\substack{f\in \mathcal{M}_n\\(f,Q)=1}}\Lambda_k(f)\right|^2 \\
&\sim \frac{q^n}{|Q|} \int_{\mathrm{U}(\deg(Q)-1)} \left|\sum_{\substack{j_1+\cdots+j_k=n\\1\leq j_1,\dots, j_k}} \mathrm{Tr}(U^{j_1})\cdots \mathrm{Tr}(U^{j_k}) \right|^2dU.
\end{align*}
\end{thm}

Meanwhile we can also consider the following two problems, which connect arithmetic questions about the $k$-fold convolutions of the von Mangoldt function to symplectic matrix integrals.

\begin{thm} \label{thm:vM-S-s}
For an irreducible polynomial $P$ of degree $2g+1$, let
\[\mathcal S^S_{\Lambda_k,n}(P) := \sum_{\substack{f \text{ monic},\, \deg f = n \\ f \equiv \square \pmod P \\ P\nmid f}} \Lambda_k(f).\]

Assume that $k\leq n$. As $q\rightarrow \infty$,
\[\mathcal S^S_{\Lambda_k,n}(P)\sim  \frac{1}{2}\sum_{\substack{f\in \mathcal{M}_n \\P\nmid f}} \Lambda_k(f) \sim \frac{q^n}{2} \binom{n-1}{k-1},\]
and
\[\frac{1}{\# \mathcal{P}_{2g+1}}\sum_{P\in \mathcal{P}_{2g+1}} \left|\mathcal{S}^S_{\Lambda_k,n}(P)-  \frac{1}{2}\sum_{\substack{f\in \mathcal{M}_n \\P\nmid f}} \Lambda_k(f)\right|^2
\sim \frac{q^n}{4}\int_{\mathrm{Sp}(2g)}\left|\sum_{\substack{j_1+\cdots+j_k=n\\1\leq j_1,\dots, j_k}} \mathrm{Tr}(U^{j_1})\cdots \mathrm{Tr}(U^{j_k})\right|^2 dU.\]
 \end{thm}

\begin{thm} \label{thm:vM-S-n}
 Let
 \[\mathcal N^S_{\Lambda_\ell,k,\nu}(v) = \sum_{\substack{f \text{ monic},\, \deg f = \nu  \\ f(0) \ne 0 \\ U(f) \in \mathrm{Sect}(v,k)}} \Lambda_\ell(f).\]
Assume that $\ell \leq \nu$.  Then as $q \rightarrow \infty$, 
\[\langle \mathcal{N}^S_{\Lambda_\ell,k,\nu}\rangle \sim q^{\nu-\kappa} \binom{\nu-1}{\ell-1},\]
 and
 \[  \frac{1}{q^\kappa}\sum_{u \in \mathbb{S}_k^1} \left|\mathcal{N}^S_{d_\ell,k,\nu}(u)- \langle \mathcal{N}^S_{\Lambda_\ell,k,\nu}\rangle\right|^2\sim\frac{q^\nu}{q^{\kappa}}\int_{\mathrm{Sp}(2\kappa-2)}\left| \sum_{\substack{j_1+\cdots+j_\ell=\nu\\1\leq j_1,\dots,j_\ell }}\mathrm{Tr}(U^{j_1})\cdots \mathrm{Tr}(U^{j_\ell}) \right|^2 dU,\]
 where $\kappa: = \left \lfloor \frac{k}{2}\right \rfloor$.
\end{thm}

As we will explore in Section \ref{sec:L}, these two functions are tractable because they appear naturally as coefficients of Dirichlet $L$-functions. More specifically, the divisor function $d_k$ appears as coefficients of the $k$th power $\mathcal{L}(u,\chi)^k$ while the von Mangoldt 
convolution $\Lambda_k$  appears as coefficients of the $k$th power of the logarithmic derivative $\left(\frac{u\mathcal{L}'(u,\chi)}{\mathcal{L}(u,\chi)}\right)^k$. In this way the divisor function and the $k$-fold von Mangoldt convolution arise naturally in the study of moments and of low level densities. 

This paper is organized as follows. Section \ref{sec:L} includes some necessary background on Dirichlet characters and $L$-functions over function fields. Sections \ref{sec:monodromy} and \ref{sec:RW} describe two problems that will lead to symplectic distributions in the limit when $q \to \infty$, as we will discuss in Sections \ref{sec:divisorsymplectic} and \ref{sec:vonMangoldtsymplectic}.  Section \ref{sec:monodromy} presents an equidistribution result of Katz, while Section  \ref{sec:RW} describes the framework of short intervals on the unit circle in the function field case. Sections \ref{sec:divisorsymplectic} and \ref{sec:vonMangoldtsymplectic} discuss averages of the divisor function and the von Mangoldt convolution, respectively; in each section, we study these averages over quadratic residues modulo a monic irreducible polynomial $P$ and over short intervals on the unit circle. Finally, Section \ref {sec:vonMangoldtunitary} treats the distribution of von Mangoldt convolutions over short intervals and over arithmetic progressions, which leads to unitary distributions in the limit when $q \rightarrow \infty$. In each of Sections  \ref{sec:divisorsymplectic}, \ref{sec:vonMangoldtsymplectic}, and \ref {sec:vonMangoldtunitary}, 
we include a subsection discussing the random matrix theory point of view, which includes the computation of some lower cases of relevant integrals.

\noindent {\bf Acknowledgments:}  We are grateful to Siegfred Baluyot, Sandro Bettin, Brian Conrey, Alessandro Fazzari, Jonathan Keating, Andean Medjedovic, Brad Rodgers, Michael Rubinstein, Zeev Rudnick, and Kannan Soundararajan for their encouragement and helpful discussions. We are particularly thankful to Brian Conrey for many helpful suggestions, including the study of the von Mangoldt convolutions in this context. We also thank Nicholas Katz for his explanations and arguments regarding the monodromy results, and Sachi Hashimoto and Wanlin Li for carefully checking the monodromy arguments. We have benefited from numerous discussions that were held at the American Institute of Mathematics and are grateful for their leadership and support.

\section{Preliminaries on Dirichlet characters and $L$-functions} \label{sec:L}
In this section we describe a general framework of $L$-functions over function fields that will be used as generating series for the arithmetic functions under consideration. For a general reference, see \cite{Rosen}. We consider the ring of polynomials $\F_q[T]$, where  $q$ is an odd prime power. Let $\mathcal{M}$ denote the set of monic polynomials in $\F_q[T]$, and $\mathcal{M}_n$ the subset of $\mathcal{M}$ of polynomials of degree $n$.  Similarly, let $\mathcal{P}$ denote the set of monic irreducible polynomials over  $\F_q[T]$, and  $\mathcal{P}_n$ the corresponding subset of irreducible polynomials of degree $n$. The norm of a polynomial $f\in \F_q[T]$ is defined as $|f|:=q^{\deg(f)}$. The zeta function for $\F_q[T]$ is given by 
\[\zeta_q(s)=\sum_{f\in \mathcal{M}} \frac{1}{|f|^s}=\prod_{P\in \mathcal{P}} \left(1-\frac{1}{|P|^s}\right)^{-1} = \frac 1{1-q^{1-s}}.\]
While the initial sum and Euler product converge for $\re(s)>1$, the right-hand side identity provides a meromorphic continuation with a single pole at $s=1$.
By applying the change of variables $u=q^{-s}$, we can also write
\[\mathcal{Z}(u)=\sum_{f\in \mathcal{M}} u^{\deg(f)} = \frac{1}{1-qu}.\]

Let $\chi$ be a Dirichlet character on  $\F_q[T]$ of conductor $m\in \F_q[T]$. $\chi$ is said to be even if $\chi(\alpha f)=\chi(f)$ for all $\alpha \in \F_q^\times$ and odd otherwise. 

The $L$-function corresponding to $\chi$ is then given by
\[L(s, \chi):=\sum_{f\in \mathcal{M}}\frac{\chi(f)}{|f|^s}=\prod_{p \in \mathcal{P}}\left(1-\frac{\chi(P)}{|P|^s}\right)^{-1},\]
which converges for $\re(s)>1$. As before, we can also write 
\[\mathcal{L}(u,\chi)=\sum_{f\in \mathcal{M}} \chi(f) u^{\deg(f)}.\]
Then $\mathcal{L}(u,\chi)$ is a polynomial of degree $\delta \leq \deg(m)-1$ 
(\cite[Proposition 4.3]{Rosen}). The Riemann hypothesis implies that \begin{equation}\label{eq:Ldet}
\mathcal{L}(u,\chi)=(1-u)^\lambda \det (1-uq^{1/2}\Theta_\chi),
\end{equation}
where $\lambda = 0$ if $\chi$ is odd and $1$ if $\chi$ is even, and $\Theta_\chi$ is a unitary matrix of dimension $N=\delta-\lambda$. 

For an $N \times N$ matrix $U$, the secular coefficients $\mathrm{Sc}_j(U)$ are defined by 
\begin{equation}\label{eq:secular}
\det (I+xU)=\sum_{j=0}^N \mathrm{Sc}_j(U)x^j.
\end{equation}
Thus, the coefficients of $\mathcal{L}(u,\chi)$ can be expressed in terms of the secular coefficients of $\Theta_\chi$. 

The divisor function arises naturally via
\begin{equation}\label{eq:dkgen}
\mathcal{L}(u,\chi)^k = \sum_{f\in \mathcal{M}} d_k(f) \chi(f) u^{\deg(f)}.
\end{equation}
Note that $\mathcal{L}(u,\chi)^k$ is a polynomial of degree $k\delta$.

We define 
\[M(n;d_k\chi)=\sum_{f \in \mathcal{M}_n} d_k(f) \chi(f).\]

\begin{lem} \label{lem:genMd}
Let $\chi$ be odd. For $n\leq k\delta$ we have  
\[M(n;d_k\chi)=(-1)^nq^{n/2} \sum_{\substack{j_1+\cdots +j_k=n\\0\leq j_1,\dots j_k\leq \delta }}\mathrm{Sc}_{j_1}(\Theta_\chi)\cdots \mathrm{Sc}_{j_k}(\Theta_\chi)\]
and $M(n;d_k\chi)=0$ otherwise.

Let $\chi$ be even. For $n\leq k\delta$, as $q \to \infty$ we have  
\[M(n;d_k\chi)=(-1)^nq^{n/2} \sum_{\substack{j_1+\cdots +j_k=n\\0\leq j_1,\dots j_k\leq \delta}}\mathrm{Sc}_{j_1}(\Theta_\chi)\cdots \mathrm{Sc}_{j_k}(\Theta_\chi)+O_{k,\delta}\left(q^\frac{n-1}{2}\right).\]
For $k\delta<n\leq k(\delta+1)$, as $q \to \infty$,
\[|M(n;d_k\chi)|\ll_{k,\delta} q^\frac{n-1}{2}.\]
Finally, $M(n;d_k\chi)=0$ for $n > k(\delta+1)$.

\end{lem}

\begin{proof}
 First consider the case of $\chi$ odd. This result is analogous to Lemma 2.2 of \cite{KR3}.

From \eqref{eq:Ldet} (with $\lambda=0$) and \eqref{eq:secular} we get 
 \[\mathcal{L}(u,\chi)^k=\sum_{\substack{0\leq j_1,\dots,j_k\leq \delta}}(-1)^{j_1+\cdots+j_k}q^{(j_1+\cdots+j_k)/2}\mathrm{Sc}_{j_1}(\Theta_{\chi})\cdots \mathrm{Sc}_{j_k}(\Theta_{\chi})u^{j_1+\cdots+j_k}.\]

Comparing coefficients with equation \eqref{eq:dkgen}, we obtain the conclusion in the case $n \leq \delta k$. If $n > k \delta$, the coefficients must be $0$, because $\deg \mathcal{L}(u,\chi)^k = k \delta$. 
 
 Now consider the case of $\chi$ even. By \eqref{eq:Ldet} (with $\lambda=1$) and \eqref{eq:secular} we get 
 \[\frac{\mathcal{L}(u,\chi)^k}{(1-u)^k}=\sum_{\substack{0\leq j_1,\dots,j_k\leq \delta}}(-1)^{j_1+\cdots+j_k}q^{(j_1+\cdots+j_k)/2}\mathrm{Sc}_{j_1}(\Theta_{\chi})\cdots \mathrm{Sc}_{j_k}(\Theta_{\chi})u^{j_1+\cdots+j_k}.\]
and
\begin{equation}\label{eq:evencirc}
\mathcal{L}(u, \chi)^k=\sum_{\substack{0\leq j_1,\dots,j_k \leq \delta \\0\leq h\leq k}}(-1)^{j_1+\cdots+j_k +h}\binom{k}{h}q^{(j_1+\cdots+j_k)/2}\mathrm{Sc}_{j_1}(\Theta_\chi)\cdots \mathrm{Sc}_{j_k}(\Theta_\chi)u^{j_1+\cdots+j_k+h}.
\end{equation}

By combining \eqref{eq:evencirc} and the definition of $M(n;d_k\chi)$,
\begin{equation}\label{eq:Meven}
M(n ;d_k\chi)=(-1)^{n}q^{n/2}\sum_{\substack{j_1+\cdots+j_k+h=n\\0\leq j_1,\dots,j_k \leq \delta\\0\leq h\leq k}}\binom{k}{h}q^{-h/2}\mathrm{Sc}_{j_1}(\Theta_\chi)\cdots \mathrm{Sc}_{j_k}(\Theta_\chi).
 \end{equation}
By the Riemann hypothesis, $|\mathrm{Sc}_{r}(\Theta_\chi)|\leq \binom{\delta-1}{r}$. Thus when $n \leq k \delta$, the dominating term in equation \eqref{eq:Meven} is the one corresponding to $h=0$, leading to the estimate in this case.

 When $k\delta < n\leq k(\delta+1)$, each term in \eqref{eq:Meven} must have $h\geq 1$, leading to the bound. Finally, the case when $k(\delta+1)<n$ follows from the bound on the degree of $\mathcal L(u,\chi)$. 
 \end{proof}

We now turn our attention to the $k$-fold convolution of the von Mangoldt function. Just as we did for the divisor function, we can consider
\begin{equation}\label{eq:Lambda-gen0}\left(\frac{u\mathcal{L}'(u,\chi)}{\mathcal{L}(u,\chi)}\right)^k = \sum_{f\in \mathcal{M}} \Lambda_k(f) \chi(f) u^{\deg(f)},
\end{equation}
where for $f$ monic,
\[\Lambda_k(f)=\sum_{\substack{f_1\cdots f_k=f \\ f_i \text{ monic}}} \Lambda(f_1)\cdots \Lambda(f_k)\]
is the convolution of the von Mangoldt function 
 \[\Lambda(f)=\begin{cases} \deg(P) & f=P^k,\\0 & \text{otherwise.} \end{cases}\]
Note that while $\mathcal{L}(u,\chi)^k$ is a polynomial,   $\left(\frac{u\mathcal{L}'(u,\chi)}{\mathcal{L}(u,\chi)}\right)^k$ is not.

We define 
\[M(n; \Lambda_k \chi)= \sum_{f \in \mathcal{M}_n} \Lambda_k(f) \chi(f).\]

Previously we could express $M(n;d_k \chi)$ as a sum of secular coefficients; here in turn we express $M(n;\Lambda_k \chi)$ as a sum of traces. 
\begin{lem} \label{lem:genML}
Let $\chi$ be odd. For $k\leq n$, we have  
\[M(n;\Lambda_k\chi)=(-1)^kq^{n/2} \sum_{\substack{j_1+\cdots +j_k=n\\1\leq j_1,\dots j_k }}\mathrm{Tr}(\Theta_\chi^{j_1})\cdots \mathrm{Tr}(\Theta_\chi^{j_k})\]
and $M(n;\Lambda_k\chi)=0$ for $1\leq n <k$.

Let $\chi$ be even. For $k\leq n$, as $q \to \infty$ we have  
\[M(n;\Lambda_k\chi)=(-1)^kq^{n/2} \sum_{\substack{j_1+\cdots +j_k=n\\1\leq j_1,\dots j_k}}
\mathrm{Tr}(\Theta_\chi^{j_1})\cdots \mathrm{Tr}(\Theta_\chi^{j_k})
+O_{n,k,\delta}\left(q^\frac{n-1}{2}\right)\]
and $M(n;d_k\chi)=0$ for $1\leq n <k$.
\end{lem}

\begin{proof}
First suppose that $\chi$ is odd.  By taking the logarithm in equation \eqref{eq:Ldet}, we have
\[\log \mathcal{L}(u,\chi)= \sum_{\theta \in \mathrm{Spec}(\Theta_{\chi})} \log(1-uq^{1/2}\theta).\]
The logarithmic derivative then gives
\[\frac{u\mathcal{L}'(u,\chi)}{\mathcal{L}(u,\chi)}=-\sum_{j=1}^\infty \sum_{\theta \in \mathrm{Spec}(\Theta_{\chi})} \theta^j q^{j/2} u^j=-\sum_{j=1}^\infty \mathrm{Tr}(\Theta_{\chi}^j) q^{j/2} u^j.\]
Finally,  the $k$-th power yields
 \[\left(\frac{u\mathcal{L}'(u,\chi_D)}{\mathcal{L}(u,\chi_D)}\right)^k=(-1)^k\sum_{n=k}^\infty \sum_{\substack{j_1+\cdots+j_k=n\\1\leq j_1,\dots, j_k}}
 \mathrm{Tr}(\Theta_{C_D}^{j_1})\cdots \mathrm{Tr}(\Theta_{C_D}^{j_k}) q^{n/2}u^n.\]
The result follows by comparing the coefficients of $u^n$ in the above equation with those of equation \eqref{eq:Lambda-gen0}.
 
Now suppose that $\chi$ is even. 
Taking the logarithm in \eqref{eq:Ldet}, we obtain, 
\[\log \mathcal{L}(u, \chi) = \log(1-u)+\sum_{\theta \in \mathrm{Spec}(\Theta_\chi)} \log(1-uq^{1/2}\theta).\]
Then taking the derivative, we get
\[\frac{u\mathcal{L}'(u,\chi)}{\mathcal{L}(u,\chi)}=-\sum_{j=1}^\infty \left(1+\sum_{\theta \in \mathrm{Spec}(\Theta_\chi)} \theta^j q^{j/2}\right) u^j=-\sum_{j=1}^\infty (1+\mathrm{Tr}(\Theta_\chi^j)
q^{j/2}) u^j.\]

Finally, we obtain for the $k$-th power,
\begin{align*}
\left(\frac{u\mathcal{L}'(u,\chi)}{\mathcal{L}(u,\chi)}\right)^k  =& (-1)^k\left(\sum_{j=1}^\infty (1+\mathrm{Tr}(\Theta_\chi^j)
q^{j/2}) u^j
\right)^k \\
=&(-1)^k\sum_{n=k}^\infty \sum_{\substack{j_1+\cdots+j_k=n\\1\leq j_1,\dots j_k}} (1+\mathrm{Tr}(\Theta_\chi^{j_1}) q^{j_1/2})\cdots (1+\mathrm{Tr}(\Theta_\chi^{j_k}) q^{j_k/2})u^n
\end{align*}
By comparing coefficients with those of equation \eqref{eq:Lambda-gen0}, we obtain
\begin{align*}\label{eq:M0lambdaU}
 M(n ;\Lambda_k\chi)
 =&(-1)^k \sum_{\substack{j_1+\cdots+j_k=n\\1\leq j_1,\dots j_k}} (1+\mathrm{Tr}(\Theta_\chi^{j_1}) q^{j_1/2})\cdots (1+\mathrm{Tr}(\Theta_\chi^{j_k}) q^{j_k/2}).
  \end{align*}
Expanding the product of traces yields the main term, $q^{n/2}\mathrm{Tr}(\Theta_\chi^{j_1}) \cdots\mathrm{Tr}(\Theta_\chi^{j_k})$; every other term in the expansion will be $O_{k,\chi,n}(q^{\frac{n-1}{2}}).$
  \end{proof}

\section{An equidistribution result}\label{sec:monodromy}

In this section we present an equidistribution result for the family of $Y^2=P(T)$, where $P\in \mathcal{P}$.  Denote by $\mathcal{H}$  the set of monic, square-free polynomials over $\F_q[T]$ and by   $\mathcal{H}_n$ the corresponding subset of polynomials of degree $n$. In this section we will let $n=2g+1$ be an odd number.

Katz and Sarnak \cite[Theorems 9.2.6, 10.1.18.3]{KS} showed that for any continuous $\C$-valued central function $F$ on the symplectic group $\mathrm{Sp}(2g)$,
\[\lim_{q \to \infty} \langle F(\Theta_Q)\rangle = \int_{\mathrm{Sp}(2g)} F(U) dU,\]
where
\[\langle F \rangle := \frac 1{\#\mathcal H_{2g+1}} \sum_{Q \in \mathcal H_{2g+1}} F(Q).\]

We will be interested in working with the family $\mathcal{P}_{2g+1}$, rather than $\mathcal H_{2g+1}$, so we need a variant of their result. In order to state the next result, we fix some notation. Let $(d_1,d_2,\dots,d_n)$ be a partition of $2g+1$, where we use the convention $d_1\geq \cdots \geq d_n$.  

We define the sets
\[\mathcal{H}_{d_1,\dots, d_n}=\left\{f\in \mathcal{H}_{2g+1}: f=\prod_i^n f_{d_i}, f_{d_i} \mbox{ monic, }\deg(f_{d_i})=d_i \right\}\]
 and 
 \[\mathcal{P}_{d_1,\dots, d_n}=\left\{f\in \mathcal{H}_{2g+1}: f=\prod_i^n f_{d_i}, f_{d_i} \mbox{ monic and irreducible }\deg(f_{d_i})=d_i \right\}.\]
That is, $\mathcal{P}_{d_1,\dots, d_n}$ is the family of monic, square-free polynomials of degree $2g+1$ with factorization pattern $(d_1,\dots,d_n)$, while $\mathcal{H}_{d_1,\dots, d_n}$ is the family of monic, square-free polynomials of degree $2g+1$ with factorization pattern $(d_1,\dots,d_n)$ as well as its refinements. More precisely, let $\rho, \sigma$ be two partitions of $2g+1$. We write  $\sigma\leq \rho$ if $\sigma$ refines $\rho$. Then 
 \[\mathcal{H}_{\rho}=\bigsqcup_{\sigma\leq \rho} \mathcal{P}_{\sigma}.\]

We will prove that the family $y^2 =f(x)$ with $f(x) \in \mathcal{P}_{d_1,\dots, d_n}$ also has monodromy group $\mathrm{Sp}(2g)$, and therefore satisfies the same equidistribution result as $\mathcal{H}_{2g+1}$. 
To do this, we will first show that the family $y^2 =f(x)$ with $f(x) \in \mathcal{H}_{d_1,\dots, d_n}$ has monodromy group $\mathrm{Sp}(2g)$.

\begin{thm}[Katz, personal communication] \label{thm:Katz-0}
Let $F$ be a continuous $\C$-valued central function on $\mathrm{Sp}(2g)$  and let $\sigma$ be any fixed partition of $2g+1$. Then 
\[\lim_{q \to \infty} \frac 1{\#\mathcal H_{\sigma}} \sum_{Q \in \mathcal H_{\sigma}} F(Q) = \int_{\mathrm{Sp}(2g)} F(U) dU.\]
 \end{thm}

\begin{proof}

We know that the monodromy group of the space of $y^2 = f(x)$ with $f(x) \in \mathcal H_{d_1, \dots, d_n}$ is a subgroup of $\mathrm{Sp}(2g)$. We would like to show that this monodromy group is in fact the full group. However, the monodromy group of a pullback is a subgroup of the original monodromy group, so it suffices to show that after an appropriate choice of pullback, we have the full monodromy group $\mathrm{Sp}(2g)$.

To do that, we fix a polynomial $f_{2g}(x) \in \mathcal H_{d_1 -1, \dots, d_n}$. Consider the space of $y^2 = f(x)$ with $f(x) = (x-t)f_{2g}(x)$, with a single parameter $t$. We restrict this space to the open subset of $t$ where $f_{2g}(t) \ne 0$. We now would like to map this restricted space to the space of $y^2=f(x)$ with $f(x) \in \mathcal H_{d_1, \dots, d_n}$, which we do by multiplying the first component by the $(x-t)$ factor. As mentioned above, the monodromy group of the space of $y^2 = f(x)$ with $f(x) = (x-t)f_{2g}(x)$ is a subgroup of the monodromy group that we would like to understand. However, by \cite[Theorem 10.1.16]{KS} (see also \cite{Yu} and \cite[Theorem 4.1]{Hall}), the monodromy group of our one-parameter space is already the full group $\mathrm{Sp}(2g)$.

This implies that the family $y^2 = f(x)$ with $f(x) \in \mathcal{H}_{d_1, \dots, d_n}$ has monodromy group $\mathrm{Sp}(2g)$. By Theorem 9.2.6 in  \cite{KS}, the large monodromy of this family implies  the equidistribution statement. 

\end{proof}

 We now consider the case of $\mathcal{P}_{d_1,\dots, d_n}$ using an inclusion-exclusion argument. 
 
 \begin{thm}\label{thm:Katz}
Let $F$ be  a continuous $\C$-valued central function on $\mathrm{Sp}(2g)$ and let $\sigma$ be any fixed partition of $2g+1$. Then 
\[\lim_{q \to \infty} \frac 1{\#\mathcal P_{\sigma}} \sum_{Q \in \mathcal P_{\sigma}} F(Q) = \int_{\mathrm{Sp}(2g)} F(U)dU.\]
 \end{thm}
 
 \begin{proof}
The cardinality of the sets $\mathcal{P}_{\sigma}$ can be estimated by
\[\#\mathcal{P}_{\sigma}=\frac{s_\sigma(m)}{m!} q^{m}+O(q^{m-1}),\]
 where $s_\sigma(m)$ denotes the number of elements with cycle pattern $\sigma$
 in the permutation group $\mathbb{S}_m$   (see for example \cite{Cohen}). In our case, $m=2g+1$.
 
 Recall that 
 \[\mathcal{H}_{\rho}=\bigsqcup_{\sigma \leq\rho} \mathcal{P}_{\sigma},\]
which implies that
 \[\# \mathcal{H}_{\rho}=\frac{q^{2g+1}}{(2g+1)!} \sum_{\sigma \leq\rho}s_{\sigma}(2g+1) +O(q^{2g}).\]
 In particular, for any two partitions $\sigma$ and $\rho$ (not necessarily related by refinement),
 \begin{equation}\label{eq:H/P}
 \lim_{q\rightarrow \infty} \frac{\# \mathcal{H}_{\rho}}{\# \mathcal{P}_{\sigma}}=\frac{1}{s_\sigma(2g+1)}\sum_{\tau \leq\rho}s_{\tau}(2g+1) .
 \end{equation}
 
Theorem \ref{thm:Katz-0} states that for a fixed factorization pattern $\rho$,
\[\lim_{q \to \infty}  \frac 1{\#\mathcal H_{\rho}} \sum_{Q \in \mathcal H_{\rho}} F(Q) = \int_{\mathrm{Sp}(2g)} F(U) dU.\]
In the formula above we have
\[\sum_{Q \in \mathcal H_{\rho}} F(Q) = \sum_{\sigma \le \rho} \sum_{Q \in \mathcal P_{\sigma}} F(Q),\]
so applying M\"obius inversion on the poset of partitions of $2g+1$ yields
\[\sum_{Q \in \mathcal P_{\sigma}} F(Q) = \sum_{\rho \le \sigma} \mu(\sigma,\rho) \sum_{Q  \in \mathcal H_{\rho}} F(Q).\]
 Plugging this expression in, applying \eqref{eq:H/P}, and expanding yields
\begin{align*}
 \lim_{q\rightarrow \infty} \frac{1}{\# \mathcal{P}_{\sigma}}\sum_{Q \in \mathcal P_{\sigma}} F(Q) =&\lim_{q\rightarrow \infty}
 \sum_{\rho\leq \sigma}\frac{\# \mathcal{H}_{\rho}}{\# \mathcal{P}_{\sigma}} \mu(\sigma,\rho)  \frac{1}{\# \mathcal{H}_{\rho}}\sum_{Q \in \mathcal H_{\rho}} F(Q)\\
 =& \frac{1}{s_\sigma(2g+1)}\sum_{\rho\leq \sigma}\sum_{\tau\leq \rho}s_{\tau}(2g+1) \mu(\sigma,\rho)\int_{\mathrm{Sp}(2g)} F(U) dU\\
 =& \frac{1}{s_\sigma(2g+1)}\sum_{\tau\leq \sigma}s_{\tau}(2g+1)  \sum_{\tau\leq \rho \leq \sigma }\mu(\sigma,\rho)\int_{\mathrm{Sp}(2g)} F(U) dU\\
  =& \frac{1}{s_\sigma(2g+1)}\sum_{\tau\leq \sigma}s_{\tau}(2g+1)  \delta(\sigma, \tau)\int_{\mathrm{Sp}(2g)} F(U) dU\\
  =&\int_{\mathrm{Sp}(2g)} F(U) dU.
\end{align*}
This concludes the proof of the statement.

\end{proof}

\section{Short intervals on the unit circle in the function field setting}\label{sec:RW}

In \cite{Rudnick-Waxman}, Rudnick and Waxman study the  distribution of angles $\theta_\mathfrak{p}$, where $\mathfrak{p}=\langle \alpha \rangle$ is an ideal of the Gaussian integers $\Z[i]$, and $\theta_\mathfrak{p}=\arg \alpha$, namely, 
\[u(\alpha):=\left(\frac{\alpha}{\overline{\alpha}}\right)^2 =e^{4i \theta_\mathfrak{p}}.\]
More precisely, they study
\[\mathcal{N}_{K,x}(\theta)=\# \{ \mathfrak{p} \mbox{ prime }, \mathrm{Norm}(\mathfrak{p})\leq x, \theta_{\mathfrak{p}} \in I_K(\theta)\},\]
 where $I_K(\theta)=[\theta-\frac{\pi}{4K}, \theta+\frac{\pi}{4K}]$, and they prove
\begin{thm}\cite[Theorem 1.1]{Rudnick-Waxman}
Assume GRH (for the family of Hecke $L$-functions). Then almost all arcs of length $1/K$ contain
at least one angle $\theta_\mathfrak{p}$ for a prime ideal with $\mathrm{Norm}(\mathfrak{p})\leq K(\log K)^{2+o(1)}$.
\end{thm}
 This is achieved by expressing $\mathrm{Var}(\mathcal{N}_{K,x})$ in terms  of zeroes of the family of Hecke $L$-functions corresponding to the characters given by $\Xi_k(\alpha)=(\alpha/\overline{\alpha})^{2k}$, $k \in \Z$ for an ideal $(\alpha)$ in  $\Z[i]$.  The result then follows from an upper bound for said variance.  The distribution of values of this family of Hecke $L$-functions was studied by Duke, Friedlander, and Iwaniec in  \cite{DFI}. Their results suggest that this family can be modeled with a symplectic regime. Indeed, Rudnick and Waxman develop a random matrix model that they use to conjecture the following. 
 \begin{conj}\label{conj:RW}\cite[Conjecture 1.2]{Rudnick-Waxman} For $1 \ll K \ll N^{1-o(1)}$
\[\mathrm{Var}(\mathcal{N}_{K,x})\sim \frac{N}{K}\min \left\{1, 2\frac{\log K}{\log N}\right\}.\]
\end{conj}
In addition, they consider the function field analogue to this problem. In order to state this problem and to recall background that we will use in later sections, we present here the function field analogue of the Gaussian integers initially developed by Bary-Soroker, Smilansky, and Wolf in \cite{BSSW}, as well as the framework considered by Rudnick and Waxman in \cite{Rudnick-Waxman}.

For $P(T)\in \mathcal{P}$, there are $A(T), B(T) \in  \F_q[T]$ with 
 \[P(T)=A(T)^2+TB(T)^2\]
 if and only if $P(0)$ is a square in $\F_q$.  Let $S:=\sqrt{-T}$. This defines an embedding $\F_q[T]\subseteq \F_q[S]$. 
Thus we can write 
\[P(T)=(A+BS)(A-BS)=\mathfrak{p}\overline{\mathfrak{p}}\]
 in  $\F_q[S]$. There are two automorphisms of $\F_q[S]$ fixing $\F_q[T]$. The nontrivial one corresponds to complex conjugation, which may be extended to the ring of formal power series:
 \[\sigma:\F_q[[S]] \rightarrow \F_q[[S]], \qquad \sigma(S)=-S.\]
The norm map is then given by 
 \[\mathrm{Norm}: \F_q[[S]]^\times \rightarrow \F_q[[T]]^\times, \qquad \mathrm{Norm}(f)=f\sigma(f)=f(S)f(-S).\]
We then have the following analogue of the unit circle:
\[\mathbb{S}^1:=\{g \in \F_q[[S]]^\times\, : \, g(0)=1, \mathrm{Norm}(g)=1\},\] 
the group of formal power series with constant term 1 and unit norm. By Hensel's Lemma the map $v \mapsto v^2$ is an automorphism of $\mathbb{S}^1$. 
 
 For $f \in  \F_q[[S]]$ let $|f|:=q^{-\mathrm{ord}(f)}$, where $\mathrm{ord}(f)=\max\{j \, :\, S^j \mid f\}$, be the absolute value associated with the place at infinity. Now consider the sectors in the circle:
\begin{equation}\label{eq:sector}
\mathrm{Sect}(v;k)=\{w\in \mathbb{S}^1\, :\, |w-v|\leq q^{-k}\}.
\end{equation}
The sector $\mathrm{Sect}(v;k)$ can be described modulo $S^k$. Indeed, 
$v \in \mathrm{Sect}(v;k)$ if and only if  $w \equiv v \pmod{S^k}$. This leads us to consider the following modular group. 
 \[\mathbb{S}^1_k:=\{f\in \F_q[S]/(S^k)\, : \, f(0)=1, \, \mathrm{Norm}(f)\equiv 1 \pmod{S^k}\}.\] 
In particular, the group $\mathbb{S}^1_k$ parametrizes the different sectors.
\begin{lem}\cite[Lemma 2.1]{Katz},  \cite[Lemma 6.1]{Rudnick-Waxman},
\begin{enumerate}
\item  The cardinality of $\mathbb{S}^1_k$ is  \[\# \mathbb{S}^1_k=q^\kappa, \mbox{ with } \kappa: = \left \lfloor \frac{k}{2}\right \rfloor.\]  
\item  We have a direct product decomposition
 \[\left(\F_q[S]/(S^k)\right)^\times = H_k\times \mathbb{S}^1_k,\]
 where 
 \[H_k:=\{ f\in \left(\F_q[S]/(S^k)\right)^\times  \, :\, f(-S) \equiv f(S) \pmod{S^k}\},\]
 and 
 \[|H_k|=(q-1)q^{\left\lfloor \frac{k-1}{2}\right \rfloor}.\]
\end{enumerate}
 \end{lem}

 Define for $f \in \F_q[S]$ coprime to $S$, 
 \[U(f):=\sqrt{ \frac{f}{\sigma(f)}}.\]
 This uses the fact that $v \mapsto v^2$ is an automorphism of $\mathbb{S}^1$ and therefore the square-root is well defined for 
 $\frac{f}{\sigma(f)}\in \mathbb{S}^1$. Notice that $U(cf)=U(f)$ for scalars $c \in \F_q^\times$. 
 
The modular counterpart of $U$ is given by
 \[U_k: \left(\F_q[S]/(S^k)\right)^\times \rightarrow \mathbb{S}^1_k, \qquad f \mapsto \sqrt{\frac{f}{\sigma(f)}} \pmod{S^k}\]
and is a surjective homomorphism whose kernel is $H_k$ (\cite[Lemma 6.2]{Rudnick-Waxman}). 
 
We are now ready to definte the analogues of the Hecke characters in this setting.  We continue to follow \cite{Rudnick-Waxman} as well as Katz \cite{Katz}. A super-even character modulo $S^k$ is a Dirichlet character 
 \[\Xi: \left(\F_q[S]/(S^k)\right)^\times \rightarrow \C^\times\]
 which is trivial on $H_k$. Therefore, super-even characters modulo $S^k$ are the characters of 
 $\left(\F_q[S]/(S^k)\right)^\times /H_k\cong \mathbb{S}_k^1$.
 
\begin{prop}{\cite[Proposition 6.3]{Rudnick-Waxman}} \label{prop:6.3}
For $f\in \left(\F_q[S]/(S^k)\right)^\times$ and $v \in \mathbb{S}_k^1$, the following are equivalent:
 \begin{enumerate}
  \item $U_k(f) \in \mathrm{Sect}(v;k)$,
  \item $U_k(f)=U_k(v)$,
  \item $fH_k=vH_k$,
  \item $\Xi(f)=\Xi(v)$ for all super-even characters $\pmod{S^k}$. 
 \end{enumerate}
 \end{prop}

 The Swan conductor of $\Xi$ is the maximal integer $d=d(\Xi)<k$ such that $\Xi$ is nontrivial on the subgroup
 \[\Gamma_d:=\left(1+(S^d)\right)/(S^k)\subset \left(\F_q[S]/(S^k)\right)^\times .\] 
Then $\Xi$ is a primitive character modulo $S^{d(\Xi)+1}$. The Swan conductor of a super-even character is necessarily odd, since these characters are automatically trivial on $\Gamma_d$ for $d$ even. 

The $L$-function associated to $\Xi$ is given by 
\[\mathcal{L}(u, \Xi)=\sum_{\substack{f \in \mathcal{M}\\f(0)\not = 0}} \Xi(f) u^{\deg(f)}=\prod_{\substack{P\in \mathcal{P}\\P(0)\not =0}} \left( 1-\Xi(P) u^{\deg(P)}\right)^{-1},\quad |u|<1/q.\]
This is a polynomial of degree $d(\Xi)$ when $\Xi$ is nontrivial. One can write
 \[\mathcal{L}(u, \Xi)=(1-u)\det\left(I-uq^{1/2}\Theta_\Xi \right)\]
 with $\Theta_\Xi \in \mathrm{U}(N)$ and $N=d(\Xi)-1$. 
 
 Katz \cite[Theorem 5.1]{Katz}  showed that for $q\rightarrow \infty$ the set of Frobenius classes 
 \[\{\Theta_\Xi\, :\, \Xi \mbox{ primitive super-even } \pmod{S^k}\}\]
 becomes uniformly distributed in $\mathrm{Sp}(2\kappa -2)$ provided that $2\kappa -2 \geq 4$, and that the same holds for $2\kappa -2 = 4$ provided that the characteristic is coprime to 10. 
 
Rudnick and Waxman study  the function that count prime ideals with a certain direction: \begin{equation}\label{eq:N-RW}\mathcal{N}_{k,\nu}(v):=\# \{ (\mathfrak{p}) \mbox{ prime }, \mathfrak{p}(0)\not = 0, \deg(\mathfrak{p})=\nu, U(\mathfrak{p}) \in \mathrm{Sect}(v,k)\},\end{equation} and prove an analogue to Conjecture \ref{conj:RW} in this setting. This rests on studying the distribution of the von Mangoldt function. It is therefore natural for us to consider von Mangoldt convolutions, in addition to the divisor function, in this context.  

\section{Symplectic averages of the divisor function} \label{sec:divisorsymplectic}
In this section we consider two different problems involving sums of the divisor function over $\F_q[T]$ which lead to symplectic distributions when $q\rightarrow \infty$.

\subsection{Average of the divisor function over the quadratic residues modulo $P$}

The first problem we consider concerns the distribution of the divisor function over quadratic residues. We first build framework for the hyperelliptic ensemble, with a particular focus on the covers defined by $Y^2=P$ with $P$ irreducible. 

Let $P(T)\in\mathcal{P}$  and $f \in \F_q[T]$. If $P\nmid f$, the quadratic residue symbol is defined by 
\[\left( \frac{f}{P}\right) \equiv f^{\frac{|P|-1}{2}} \pmod{P}.\]
If $P\mid f$, we set $\left( \frac{f}{P}\right) =0$. If $Q=P_1^{e_1}\cdots P_r^{e_r}$ with each $P_j$ irreducible, then the Jacobi symbol is given by 
\[\left( \frac{f}{Q}\right) =\prod_{j=1}^r \left(\frac{f}{P_j}\right)^{e_j}.\]
From now on we will assume that $q\equiv 1 \pmod{4}$; in this case, quadratic reciprocity implies $\left( \frac{A}{B}\right)=\left( \frac{B}{A}\right)$ for any $A,B\in \mathcal{M}$ non-zero such that $(A,B)=1$.

For $D \in \mathcal{H}$, we consider the quadratic character
\[\chi_D(f)=\left(\frac{D}{f}\right).\]

If we consider the hyperelliptic curve with model $C_D: y^2=D$, then the zeta function associated to $C_D$ is given by 
\[\mathcal{Z}_{C_D}=\exp \left(\sum_{n=1}^\infty N_n(C_D) \frac{u^n}{n}\right).\]
The Weil conjectures then imply that 
\[\mathcal{Z}_{C_D}=\frac{\mathcal{L}^*(u,\chi_D)}{(1-u)(1-qu)},\]
where $\mathcal{L}^*(u,\chi_D)$ is the completed $L$-function given by 
\[\mathcal{L}(u,\chi_D)=(1-u)^\lambda\mathcal{L}^*(u,\chi_D)\]
with $\lambda=1$ if $\deg(D)$ even, and 0 if $\deg(D)$ odd. We also have that 
$\mathcal{L}^*(u,\chi_D)$ is a polynomial of even degree $\deg(D)-1-\lambda$.

For simplicity we will restrict ourselves to the case $\deg(D)$ odd (so the character $\chi_D$ is odd) and we will write $\deg(D)=2g+1$, where $g$ is the genus of $C_D$. This does not restrict the family if we think in terms of hyperelliptic covers, since we can always find a model with $D$ of degree $2g+1$. Therefore we can also assume that $\mathcal{L}^*(u,\chi_D)=\mathcal{L}(u,\chi_D)$. 

By the Riemann Hypothesis all the zeroes of $\mathcal{L}^*(u,\chi_D)$ satisfy $|u|=\frac{1}{\sqrt{q}}$. Thus, we can write 
\[\mathcal{L}^*(u,\chi_D)=\det(I-u q^{1/2}\Theta_{C_D})\]
where $\Theta_{C_D} \in \mathrm{Sp}(2g)$. Its conjugacy class is known as the unitarized Frobenius class. 

Lemma \ref{lem:genMd} in this case gives the following statement. 
\begin{lem}\label{lem:M}
 For $D \in \mathcal{H}_{2g+1}$, and 
 \[M(n; d_k \chi_D) = \sum_{f \in \mathcal{M}_n} d_k(f) \chi_D(f),\]
 we have for $n \leq 2gk$,
 \[M(n; d_k \chi_D) =(-1)^nq^{n/2}\sum_{\substack{j_1+\cdots+j_k=n\\0\leq j_1,\dots,j_k\leq 2g}}\mathrm{Sc}_{j_1}(\Theta_{C_D})\cdots \mathrm{Sc}_{j_k}(\Theta_{C_D}),\]
and $M(n; d_k \chi_D)=0$ otherwise. 
 \end{lem}

We could then use the connection between $M(n;d_k\chi_D)$ and secular coefficients to study the distribution of $d_k(f)$ among quadratic residues modulo $D$. However, for an arbitrary $D$, $\chi_D$ does not directly detect quadratic residues; we would need to rely on an inclusion-exclusion argument as well. Specifically, we would want to consider the sum
\[\frac{1}{2^{\Omega(D)}}\sum_{D_0\mid D}M(n; d_k \chi_{D_0}\chi_D),\]
which then becomes unwieldy in the general case. Instead, we restrict to the case when $D$ is a monic irreducible polynomial $P$, so that this sum becomes much simpler. In exchange, we end up with a sum over irreducible polynomials $P$ of $M(n;d_k\chi_P)$. In particular, in order to connect our sum to a symplectic matrix integral, we will need an equidistribution result about $\Theta_{C_P}$ for $P$ ranging over monic irreducible polynomials, instead of one about $\Theta_{C_D}$, for $D$ ranging over monic square-free polynomials. This result was proven in Section \ref{sec:monodromy}. Notice that to achieve this, we returned to an inclusion-exclusion argument at the end of our computation.

Let $P \in \mathcal{P}_{2g+1}$. We consider the question of computing the mean and variance of the following function as $q \rightarrow \infty$: 
\[\mathcal{S}^S_{d_k,n}(P)=\sum_{\substack{f\in \mathcal{M}_n \\f\equiv \square \pmod{P}\\P\nmid f}} d_k(f),\]
where the $\square$ denotes a square. In other words, the sum take places over the quadratic residues modulo $P$ of fixed degree $n$. 

Recalling that $P\nmid f$ is a square modulo $P$ iff $1+\chi_P(f)=2$, we obtain
\begin{equation}\label{eq:sumofN}
\mathcal{S}^S_{d_k,n}(P)=\frac{1}{2}\sum_{\substack{f\in \mathcal{M}_n \\P\nmid f}} d_k(f)+\frac{1}{2}M(n; d_k \chi_P).
\end{equation}

The main term of $\mathcal{S}^S_{d_k,n}(P)$ comes from the first sum in \eqref{eq:sumofN}. 
\begin{lem}\label{lem:main-average-divisor} We have that
\[\mathcal{S}^S_{d_k,n}(P) = \frac{1}{2}\sum_{\substack{f\in \mathcal{M}_n \\P\nmid f}} d_k(f) \left(1+O\left(\frac{1}{q}\right)\right)=\frac{q^n}{2}\binom{k+n-1}{k-1}+O(q^{n-1}).\] 
\end{lem}
\begin{proof}
We can estimate the first term in \eqref{eq:sumofN} by considering its generating function:
\begin{align*}
\sum_{\substack{f\in \mathcal{M} \\P\nmid f}} d_k(f)u^{\deg(f)}
=& \left(\mathcal{Z}(u)(1-u^{\deg(P)})\right)^k=\left(\frac{1-u^{2g+1}}{1-qu}\right)^k\\
=& \sum_{n=0}^\infty \sum_{m=0}^k \binom{k}{m}\binom{-k}{n-(2g+1)m} q^{n-(2g+1)m}(-1)^nu^n.
\end{align*}
By focusing on the coefficient of $u^n$, we obtain 
\[\sum_{\substack{f\in \mathcal{M}_n \\P\nmid f}} d_k(f)=(-1)^n\sum_{m=0}^k \binom{k}{m}\binom{-k}{n-(2g+1)m} q^{n-(2g+1)m}=q^n\binom{k+n-1}{k-1}+O(q^{n-1}).\]
The above estimate is independent of $P$.

For the second term in \eqref{eq:sumofN}, the Riemann Hypothesis gives 
\[\left|M(n; d_k \chi_P)\right| \ll q^{n/2}.\]
\end{proof}

We consider the variance
\begin{align*}
 \mathrm{Var}(\mathcal{S}^S_{d_k,n})=& \frac{1}{\# \mathcal{P}_{2g+1}}\sum_{P\in \mathcal{P}_{2g+1}} \left|\mathcal{S}^S_{d_k,n}(P)- \frac{1}{2}\sum_{\substack{f\in \mathcal{M}_n \\P\nmid f}} d_k(f)\right|^2,
\end{align*}
which by Lemma \ref{lem:main-average-divisor} is given by
\begin{equation}\label{eq:variance}
\mathrm{Var}(\mathcal{S}^S_{d_k,n})=
 \frac{1}{4\#\mathcal{P}_{2g+1}}\sum_{\substack{P\in\mathcal{P}_{2g+1}}} |M(n; d_k \chi_P)|^2 \left(1+O\left(\frac{1}{q}\right)\right).
 \end{equation}

 By combining equation \eqref{eq:variance} and Lemma \ref{lem:M}, we obtain
\begin{equation}\label{eq:var1}
 \mathrm{Var}(\mathcal{S}^S_{d_k,n})
 = \frac{q^n}{4\# \mathcal{P}_{2g+1}}\sum_{\substack{P\in\mathcal{P}_{2g+1}}}\left| \sum_{\substack{j_1+\cdots+j_k=n\\0\leq j_1,\dots,j_k\leq 2g}}\mathrm{Sc}_{j_1}(\Theta_{C_P})\cdots \mathrm{Sc}_{j_k}(\Theta_{C_P})\right|^2 \left(1+O\left(\frac{1}{q}\right)\right).
\end{equation}
We set $\sigma=(2g+1)$ in Theorem \ref{thm:Katz} and  conclude that
\begin{align*}
 \lim_{q\rightarrow \infty} \frac{1}{\# \mathcal{P}_{2g+1}}\sum_{Q \in \mathcal P_{2g+1}} F(Q) 
 =&\int_{\mathrm{Sp}(2g)} F(U) dU.
\end{align*}
Applying this limit to equation \eqref{eq:var1}, we obtain the statement of  Theorem \ref{thm:var}. 
\begin{thm}
 Let $n\leq 2gk$. As $q \rightarrow \infty$ 
 \begin{align*}
\mathrm{Var}(\mathcal{S}^S_{d_k,n})
 \sim & \frac{q^n}{4} \int_{\mathrm{Sp}(2g)}\left| \sum_{\substack{j_1+\cdots+j_k=n\\0\leq j_1,\dots,j_k\leq 2g}}\mathrm{Sc}_{j_1}(U)\cdots \mathrm{Sc}_{j_k}(U)\right|^2 dU.
\end{align*}
 \end{thm}

\subsection{Average of the divisor function over short arcs on the unit circle}

The second problem we study is the variance of the following sum:
 \[\mathcal{N}^S_{d_\ell,k,\nu}(v)=\sum_{\substack{f \in \mathcal{M}_\nu\\ f(0)\not =0\\U(f)\in \mathrm{Sect}(v,k)}} d_\ell(f).\]
Many of the ideas and techniques that we discuss here come from \cite{Rudnick-Waxman}. 

As in Section \ref{sec:L} we set 
\[M_0(\nu ;d_\ell\Xi)=\sum_{\substack{f\in \mathcal{M}_{\nu}\\f(0)\not = 0}} d_\ell(f) \Xi(f),\]
where the subindex $0$ indicates that we perform the sum over $f(0)\not = 0$. 

Lemma \ref{lem:genMd} becomes the following statement in this setting. 
\begin{lem}\label{lem:secM0} We have for 
$\nu \leq \ell(d(\Xi)-1)$, 
\[M_0(\nu ;d_\ell\Xi)=(-1)^{\nu}q^{\nu/2}\sum_{\substack{j_1+\cdots+j_\ell=\nu\\0\leq j_1,\dots,j_\ell \leq d(\Xi)-1}}\mathrm{Sc}_{j_1}(\Theta_{\Xi})\cdots \mathrm{Sc}_{j_\ell}(\Theta_{\Xi})+O_{\ell,k}\left(q^{\frac{\nu-1}{2}}\right).\]

If $\ell(d(\Xi)-1)< \nu \leq \ell d(\Xi)$, 
\[|M_0(\nu ;d_\ell\Xi)|\ll_{\ell, k} q^{\frac{\nu-1}{2}}.\]

Finally, if $\ell d(\Xi) <\nu$,   $M_0(\nu ;d_\ell\Xi)=0$. 
 \end{lem}

We start our analysis by looking at  the mean value of $\mathcal N^S_{d_\ell,k,\nu}$ averaged over all the directions of $v\in \mathbb{S}_k^1$. 

 \begin{lem} 
  We have that 
  \[\langle \mathcal{N}^S_{d_\ell,k,\nu}\rangle=
  \frac{1}{q^\kappa} \sum_{v \in \mathbb{S}_k^1} \mathcal{N}^S_{d_\ell,k,\nu}(v)=\frac{1}{q^\kappa} \sum_{\substack{f \in \mathcal{M}_\nu\\ f(0)\not =0}} d_\ell(f)  =q^{\nu-\kappa}\binom{\ell+\nu-1}{\ell-1} +O\left(q^{\nu-\kappa-1}\right).\]
 \end{lem}

\begin{proof} Notice that we can count polynomials with the condition $f(0)\not =0$ by removing the factor $(1-u)^{-1}$ from the Euler product of the zeta function:
 \[\mathcal{Z}(u)^\ell (1-u)^\ell = \sum_{\substack{f\in \mathcal{M}\\f(0)\not = 0}} d_\ell(f) u^{\deg(f)}.\]
 We use the formula for the zeta function to find a closed expression for the left-hand side term. 
  \[\mathcal{Z}(u)^\ell (1-u)^\ell =\left(\frac{1-u}{1-qu}\right)^\ell = \sum_{\nu=0}^\infty \sum_{m=0}^\ell \binom{\ell}{m}\binom{-\ell}{\nu-m}q^{\nu-m} (-1)^\nu u^\nu.\]
The result follows by comparing   the coefficients of $u^\nu$.
  \end{proof}

Our next goal is to  obtain a formula for $\mathcal{N}^S_{d_\ell,k,\nu}(v)$ in terms of the super-even characters $\Xi$. By Proposition \ref{prop:6.3} and the orthogonality relations, we find, for $f \in \mathcal{M}_\nu$,
\[\frac{1}{q^{\kappa}} \sum_{\Xi \text{ super-even} \pmod{S^k}} \overline{\Xi(v)}\Xi(f) =\begin{cases} 1 & U(f)\in \mathrm{Sect}(v;k),\\ 0 & \text{otherwise}.
\end{cases}\]
Hence
\[\mathcal{N}^S_{d_\ell,k,\nu}(v)=\sum_{\substack{f \in \mathcal{M}_\nu\\ f(0)\not =0\\U(f)\in \mathrm{Sect}(u,k)}} d_\ell(f)= \frac{1}{q^{\kappa}}\sum_{\Xi \text{ super-even} \pmod{S^k}} \overline{\Xi(u)} \sum_{\substack{f \in \mathcal{M}_\nu\\ f(0)\not =0}} d_\ell(f)\Xi(f).\]
The contribution from the trivial character $\Xi_0$ is precisely \[\frac{1}{q^{\kappa}}\sum_{\substack{f \in \mathcal{M}_\nu\\ f(0)\not =0}} d_\ell(f)= \langle \mathcal{N}^S_{d_\ell,k,\nu}\rangle.\] 
Thus
\begin{align}
\mathcal{N}^S_{d_\ell,k,\nu}(v)- \langle \mathcal{N}^S_{d_\ell,k,\nu}\rangle=&\frac{1}{q^{\kappa}}\sum_{\substack{\Xi \text{ super-even} \pmod{S^k}\\ \Xi\not = \Xi_0}} \overline{\Xi(v)} \sum_{\substack{f \in \mathcal{M}_\nu\\ f(0)\not =0}} d_\ell(f)\Xi(f)\nonumber \\
=&\frac{1}{q^{\kappa}}\sum_{\substack{\Xi \text{ super-even} \pmod{S^k}\\ \Xi\not = \Xi_0}}  \overline{\Xi(v)} M_0(\nu; d_\ell \Xi). \label{eq:sumM0}
\end{align}

Recall that we want to compute the variance
\begin{align}\label{eq:var-S-dk}
 \mathrm{Var}(\mathcal{N}^S_{d_\ell,k,\nu})=& \frac{1}{q^\kappa}\sum_{u \in \mathbb{S}_k^1} \left|\mathcal{N}^S_{d_\ell,k,\nu}(u)- \langle \mathcal{N}^S_{d_\ell,k,\nu}\rangle\right|^2.
\end{align}

By applying the orthogonality relations
\begin{equation}\label{eq:orthogonalitysupereven}\sum_{u \in \mathbb{S}_k^1} \overline{\Xi_1(u)}\Xi_2(u)=\begin{cases}
                                                           q^\kappa & \Xi_1=\Xi_2,\\
                                                           0 & \text{otherwise},
                                                          \end{cases}
                                \end{equation} to equations \eqref{eq:sumM0} and \eqref{eq:var-S-dk}, we obtain
 \begin{align*}
 \mathrm{Var}(\mathcal{N}^S_{d_\ell,k,\nu})=& \frac{1}{q^\kappa}\sum_{u \in \mathbb{S}_k^1}
 \frac{1}{q^{2\kappa}}\sum_{\substack{\Xi_1, \Xi_2 \text{ super-even} \pmod{S^k}\\ \Xi_1, \Xi_2\not = \Xi_0}}  \overline{\Xi_1(u)} M_0(\nu; d_\ell \Xi_1) \Xi_2(u) \overline{M_0(\nu; d_\ell \Xi_2)}\\
 =&
 \frac{1}{q^{2\kappa}}\sum_{\substack{\Xi_1, \Xi_2 \text{ super-even} \pmod{S^k}\\ \Xi_1, \Xi_2\not = \Xi_0}} M_0(\nu; d_\ell \Xi_1) \overline{M_0(\nu; d_\ell \Xi_2)} \frac{1}{q^\kappa}\sum_{u \in \mathbb{S}_k^1} \overline{\Xi_1(u)}\Xi_2(u)  \\
 =&
 \frac{1}{q^{2\kappa}}\sum_{\substack{\Xi\text{ super-even} \pmod{S^k}\\ \Xi\not = \Xi_0}} |M_0(\nu; d_\ell \Xi)|^2.
\end{align*}

Combining the above with Lemma \ref{lem:secM0} yields
\begin{equation}\label{eq:varsupereven}
\mathrm{Var}(\mathcal{N}^S_{d_\ell,k,\nu})
 = \frac{q^\nu}{q^{2\kappa}}\sum_{\substack{\Xi\text{ super-even} \pmod{S^k}\\ \Xi\not = \Xi_0}}
\left| \sum_{\substack{j_1+\cdots+j_\ell=\nu\\0\leq j_1,\dots,j_\ell \leq d(\Xi)-1}}\mathrm{Sc}_{j_1}(\Theta_{\Xi})\cdots \mathrm{Sc}_{j_\ell}(\Theta_{\Xi})\right|^2\left(1+O\left(\frac{1}{\sqrt{q}}\right)\right).
\end{equation}

We are now ready to prove Theorem \ref{thm:sector}.
\begin{thm} \label{thm:sector-text}
 Let $\nu \leq \ell (2\kappa-2)$ with $\kappa=\left\lfloor \frac{k}{2}\right \rfloor$. As $q\rightarrow \infty$, 
\[\mathrm{Var}(\mathcal{N}^S_{d_\ell,k,\nu})\sim\frac{q^\nu}{q^{\kappa}}\int_{\mathrm{Sp}(2\kappa-2)}\left| \sum_{\substack{j_1+\cdots+j_\ell=\nu\\0\leq j_1,\dots,j_\ell \leq 2\kappa-2}}\mathrm{Sc}_{j_1}(U)\cdots \mathrm{Sc}_{j_\ell}(U)\right|^2 dU.\]
\end{thm}
\begin{proof}
In equation \eqref{eq:varsupereven}, we separate the characters according to their Swan conductor, which is necessarily an odd integer $d(\Xi)<k$ with maximal value $2\kappa-1$. The characters with maximal conductor are primitive; the contribution from the others is negligible. Thus, we can consider the sum only over the primitive characters, and the  result follows from Katz \cite[Theorem 5.1]{Katz}.
\end{proof}

\subsection{Relationship with Random Matrix Theory}

Our goal here is to discuss what is known about the integral
\begin{equation}\label{eq:int-dk-symplectic}
I_{d_k}^S(n;N):=\int_{\mathrm{Sp}(2N)} \sum_{\substack{j_1+\cdots+j_k=n\\0\leq j_1,\dots,j_k \leq 2N}}\mathrm{Sc}_{j_1}(U)\cdots \mathrm{Sc}_{j_k}(U) dU.
\end{equation}
Observe that the integral $I_{d_k}^S(n;N)$ given by  \eqref{eq:int-dk-symplectic} measures the discrepancy of the objects discussed previously, rather than the variance. However, from the point of view of random matrix theory, $I_{d_k}^S(n;N)$ is a more natural object to consider. The methods used to prove Theorem \ref{thm:var} apply to obtain a similar statement for the discrepancy. 

\begin{thm}\label{thm:discrepancy-dk}
 Assume that $n\leq 2gk$. As $q\rightarrow \infty$, 
\begin{align*}
\frac{1}{\# \mathcal{P}_{2g+1}}\sum_{P\in \mathcal{P}_{2g+1}} &\left(\mathcal{S}^S_{d_k,n}(P) -\frac{1}{2}\sum_{\substack{f\in \mathcal{M}_n \\P\nmid f}} d_k(f) \right)
\\
&\sim \frac{(-1)^nq^{n/2}}{2}\int_{\mathrm{Sp}(2g)}\sum_{\substack{j_1+\cdots+j_k=n\\0\leq j_1,\dots, j_k\leq 2g}} \mathrm{Sc}_{j_1}(U)\cdots \mathrm{Sc}_{j_k}(U)  dU.
\end{align*}
\end{thm}

\begin{rem} \label{rem:the mean is mean}
Note that $\frac{1}{2}\sum_{\substack{f\in \mathcal{M}_n \\P\nmid f}} d_k(f)$ is the main term of $\mathcal{S}^S_{d_k,n}(P)$, so Theorem \ref{thm:discrepancy-dk} picks up on a contribution from the second-order terms. In the case of $\mathcal N_{d_\ell,k,\nu}^S(v)$, $\langle \mathcal N_{d_\ell,k,\nu}^S \rangle$ is the limiting mean value rather than the main term, so the discrepancy above goes to zero as $q \rightarrow \infty$.
\end{rem}

In \cite{AndyMike} Medjedovic and Rubinstein prove the following result. 

\begin{thm}\cite[Theorem 7]{AndyMike} \label{thm:AndyMike}
Let \[P_{k,N}(x)=\sum_{n=0}^{2kN} I_{d_k}^S(n,N)x^n.\]
Then 
\[P_{k,N}(x)=\frac{1}{(1-x^2)^{\binom{k+1}{2}}}\det_{1\leq i, j \leq k} \left[\binom{j-1}{i-1}x^{j-i}-\binom{2N+2k+1-j}{i-1}x^{2N+2k+2-j-i}\right].\]
\end{thm}

\begin{cor}
If $k = 2$, then $I_{d_2}^S(n,N)$ is given by
\[\int_{\mathrm{Sp}(2N)} \sum_{\substack{j_1 + j_2 = n \\ 0 \le j_1, j_2 \le 2N}} \mathrm{Sc}_{j_1}(U)\mathrm{Sc}_{j_2}(U) dU = \begin{cases}
              \binom{\frac{n}{2}+2}{2}\eta_n & 0\leq n\leq N,\\
               \binom{2N-\frac{n}{2}+2}{2}\eta_n & N+1\leq n \leq 2N,
             \end{cases}\]
where
\begin{equation}\label{eq:eta}\eta_j=\begin{cases}
             1 & j \mbox{ even}, \\
              0 & j \mbox{ odd}. \\
            \end{cases}
            \end{equation}
            
As $N \to \infty$, $I_{d_2}^S(n,N)$ is asymptotic to
\[I_{d_2}^S(n,N) \sim \gamma_{d_2}^S(c)N^2,\]
where $c = n/N$ and 
\[\gamma_{d_2}^S(c)=\begin{cases} 
                 \frac{c^2}{2} & 0\leq c \leq 1,\\
                 \frac{(2-c)^2}{2} & 1\leq c \leq 2.
                \end{cases}
\]
\end{cor}
\begin{proof}
If $k=2$, Theorem \ref{thm:AndyMike} gives
\begin{align*}
P_{2,N}(x)=&\frac{1-x^{4N+6}-(2N+3)x^{2N+2}(1-x^2)}{(1-x^2)^3}\\
=&\frac{\sum_{\ell=0}^{2N+2} x^{2\ell} -(2N+3)x^{2N+2}}{(1-x^2)^2}\\
=&\sum_{\ell=0}^N \binom{\ell+2}{2} x^{2\ell} +\sum_{\ell=N+1}^{2N} \binom{2N-\ell+2}{2} x^{2\ell}.
\end{align*}

By looking at the coefficient of $x^{2\ell}$, we obtain the statement.  
\end{proof}

As noted in Remark \ref{rem:the mean is mean}, the discrepancy may be negligible in some cases, forcing us to study the variance. From the random matrix theory point of view, this leads to the study of the integral 
\begin{equation}\label{eq:int-dk-symplectic-square}
I_{d_k,2}^S(n;N):=\int_{\mathrm{Sp}(2N)} \left|\sum_{\substack{j_1+\cdots+j_k=n\\0\leq j_1,\dots,j_k \leq 2N}}\mathrm{Sc}_{j_1}(U)\cdots \mathrm{Sc}_{j_k}(U)\right|^2 dU.
\end{equation}
Here we will compute  $I_{d_1,2}^S(n;N)$. The following result is due to Conrey, Farmer, Keating, Rubinstein, and Snaith \cite{CFKRS-Autocorrelations}. We state here the version of Bump and Gamburd. 
 \begin{prop}\cite[Proposition 11]{Bump-Gamburd}
\begin{equation}\label{eq:BG}
\int_{\mathrm{Sp}(2N)} \prod_{j=1}^k \det (I+x_jU)dU=\sum_{\varepsilon \in \{-1,1\}} \prod_{j=1}^k x_j^{N(1-\varepsilon_j)}\prod_{i\leq j} (1-x_i^{\varepsilon_i} x_j^{\varepsilon_j})^{-1}.
\end{equation}
 \end{prop}
 
 \begin{cor}
 If $k = 1$, then $I_{d_1,2}^S(n,N)$ is given by
 \[\int_{\mathrm{Sp}(2N)}\left|\mathrm{Sc}_n(U)\right|^2 dU =\begin{cases}
             \left \lfloor \frac{n+2}{2} \right\rfloor & 0\leq n\leq N,\\
              \left \lfloor \frac{2N-n+2}{2} \right\rfloor  & N+1\leq n \leq 2N.
             \end{cases}\]
As $N \to \infty$, $I^S_{d_1,2}(n,N)$ is asymptotic to 
\[I^S_{d_1,2}(n,N) \sim \gamma_{d_1,2}^S(c) N,\]
where $c = n/N$ and
   \[\gamma_{d_1,2}^S(c)=\begin{cases}
             \frac{c}{2}  & 0\leq c\leq 1,\\
             \frac{2-c}{2}  & 1\leq c \leq 2.
             \end{cases}\]
 \end{cor}
 \begin{proof}

In order to compute $I_{d_1,2}^S(n;N)$, we consider the case of two variables in order to obtain 
  \begin{align*}
  \int_{\mathrm{Sp}(2N)} &\det (I+x_1U)\det (I+x_2U)dU\\
  =&\frac{1-x_1^{2N+3}x_2^{2N+3}}{(1-x_1x_2)(1-x_1^2)(1-x_2^2)}-\frac{x_1^{2N+3}-x_2^{2N+3}}{(x_1-x_2)(1-x_1^2)(1-x_2^2)}\\
     =& \frac{\sum_{m=0}^{N} x_1^m(x_2^m-x_2^{2N+2-m})+\sum_{m=N+2}^{2N+2}x_1^m(x_2^m-x_2^{2N+2-m})}{(1-x_1^2)(1-x_2^2)}\\
     =& \frac{\sum_{m=0}^{N} x_1^m(x_2^m-x_2^{2N+2-m})+\sum_{\ell=0}^{N}x_1^{2N+2-\ell}(x_2^{2N+2-\ell}-x_2^{\ell})}{(1-x_1^2)(1-x_2^2)}\\
          =& \frac{\sum_{m=0}^{N} (x_1^m -x_1^{2N+2-m})  (x_2^m-x_2^{2N+2-m})}{(1-x_1^2)(1-x_2^2)}\\
          =& \sum_{m=0}^{N} x_1^mx_2^m\left(\sum_{\ell_1=0}^{N-m}x_1^{2\ell_1}  \sum_{\ell_2=0}^{N-m}x_2^{2\ell_2}\right). 
\end{align*}
$I_{d_1,2}^S(n;N)$ is then given by the coefficient of $x_1^nx_2^n$, which gives the corollary. 
\end{proof}
In the above proof, the fact that the exponents of $x_1$ and $x_2$ are equal guarantees that we integrate $\mathrm{Sc}_n(U)^2$, and not just any mixed product.
On the other hand, the method described above to obtain $I_{d_1,2}^S(n;N)$ does not extend to $I_{d_k,2}^S(n;N)$ in an obvious way. The integral $I_{d_k,2}^S(n;N)$ remains a problem to be explored.

\section{Symplectic averages of the von Mangoldt convolution} \label{sec:vonMangoldtsymplectic}

In this section we study the mean and variance of quantities that are very similar to $\mathcal{S}^S_{d_k,n}$ and $\mathcal{N}^S_{d_\ell,k,\nu}$ from Section \ref{sec:divisorsymplectic}, but with the convolution of the von Mangoldt function $\Lambda_k$ in place of the divisor function $d_k$.
 
 \subsection{Average of the von Mangoldt convolution function over the quadratic residues modulo $P$}
 
Let $P \in \mathcal{P}_{2g+1}$. We consider the question of studying the distribution of the von Mangoldt convolution over square residues modulo $P$, as $q \rightarrow \infty$: 
\[\mathcal{S}^S_{\Lambda_k,n}(P)=\sum_{\substack{f\in \mathcal{M}_n \\f\equiv \square \pmod{P}\\P\nmid f}} \Lambda_k(f).\]

Our goal will be to understand the mean value and variance of $\mathcal S^S_{\Lambda_k,n}$.

Let $D\in \mathcal{H}_{2g+1}$ as before. We also consider
\[M(n;\Lambda_k\chi_D)=\sum_{f\in \mathcal{M}_n} \Lambda_k(f)\chi_D(f).\]
  Lemma \ref{lem:genML} becomes the following statement. 
 \begin{lem} \label{lem:M-vM} Let $D\in \mathcal{H}_{2g+1}$. For $k\leq n$, we have
\[
 M(n;\Lambda_k\chi_D)= (-1)^k q^{n/2}\sum_{\substack{j_1+\cdots+j_k=n\\1\leq j_1,\dots, j_k}} \mathrm{Tr}(\Theta_{C_D}^{j_1})\cdots \mathrm{Tr}(\Theta_{C_D}^{j_k}),
 \]
 and $M(n;\Lambda_k\chi_D)=0$ for $1\leq n < k$. 
 \end{lem}

   As in the case of $d_k$, we have, for $\deg(P)=2g+1$,  
   \begin{equation}\label{eq:sumofN-lambda}
\mathcal{S}^S_{\Lambda_k,n}(P)=\frac{1}{2}\sum_{\substack{f\in \mathcal{M}_n \\P\nmid f}} \Lambda_k(f)+\frac{1}{2}M(n; \Lambda_k \chi_P).
\end{equation}
The main term of $\mathcal{S}^S_{\Lambda_k,n}(P)$ comes from the first term in \eqref{eq:sumofN-lambda}.
\begin{lem}\label{lem:vM-S} As $q \to \infty$,
\[\mathcal{S}^S_{\Lambda_k,n}(P)=
 \frac{1}{2}\sum_{\substack{f\in \mathcal{M}_n \\P\nmid f}} \Lambda_k(f)\left(1+O\left(\frac{1}{q}\right)\right)=\frac{q^n}{2} \binom{n-1}{k-1} +O(q^{n-1}).\]
\end{lem}
\begin{proof}
We estimate the first term in \eqref{eq:sumofN-lambda} by considering its generating function:
\begin{align*}
\sum_{\substack{f\in \mathcal{M} \\P\nmid f}} \Lambda_k(f)u^{\deg(f)}=&\left(\frac{u\mathcal{Z}'(u)}{\mathcal{Z}(u)}-\frac{\deg(P)u^{\deg(P)}}{1-u^{\deg(P)}}\right)^k=\left(\frac{qu}{1-qu}-\frac{(2g+1)u^{2g+1}}{1-u^{2g+1}}\right)^k\\
=& \sum_{j=0}^k \binom{k}{j} q^j u^j (-1)^{k-j} (2g+1)^{k-j}u^{(2g+1)(k-j)} (1-qu)^{-j} (1-u^{2g+1})^{j-k}\\
=& \sum_{n=0}^\infty  \sum_{j=0}^k \binom{k}{j} q^j (-1)^{k-j} (2g+1)^{k-j} \sum_{\substack{m=0\\((k-j)+m)(2g+1)+j \leq n}} \binom{j-k}{m} (-1)^m\\&\times  \binom{-j}{n-((k-j)+m)(2g+1)-j} (-q)^{n-((k-j)+m)(2g+1)-j} u^n.
\end{align*}
Taking the coefficient of $u^n$ yields a formula for the sum of $\Lambda_k$ over monic polynomials of fixed degree $n$, coprime to a fixed $P$:
\begin{align}\label{eq:lambdapnd}
 \sum_{\substack{f\in \mathcal{M}_n \\P\nmid f}} \Lambda_k(f)= \sum_{j=0}^k & \binom{k}{j}  (2g+1)^{k-j}\sum_{\substack{m=0\\((k-j)+m)(2g+1)+j \leq n}}  \binom{j-k}{m}\nonumber \\&\times \binom{-j}{n-((k-j)+m)(2g+1)-j} (-1)^{n-j}  q^{n-((k-j)+m)(2g+1)}.
\end{align}
To maximize the power of $q$, we must take $j=k$ (which implies $m=0$), thus leading to 
\begin{align*}
 \sum_{\substack{f\in \mathcal{M}_n \\P\nmid f}} \Lambda_k(f)=& q^n (-1)^{n-k} \binom{-k}{n-k} +O(q^{n-1})=q^n \binom{n-1}{k-1} +O(q^{n-1}).
\end{align*}

Now we examine the second term in \eqref{eq:sumofN-lambda}. By Lemma \ref{lem:M-vM}, we deduce 
\[|M(n; \Lambda_k \chi_P)| \ll \binom{n-1}{k-1} (2g)^k q^{n/2}.\]
\end{proof}

We can now turn to the problem of finding the variance
\begin{align*}
 \mathrm{Var}(\mathcal{S}^S_{\Lambda_k,n})=& \frac{1}{\# \mathcal{P}_{2g+1}}\sum_{P\in \mathcal{P}_{2g+1}} \left|\mathcal{S}^S_{\Lambda_k,n}(P)-  \frac{1}{2}\sum_{\substack{f\in \mathcal{M}_n \\P\nmid f}} \Lambda_k(f)\right|^2.
\end{align*}

By the previous discussion the variance is given by 
\begin{equation}\label{eq:variancelambda}
\mathrm{Var}(\mathcal{S}^S_{\Lambda_k,n})=
 \frac{1}{4\#\mathcal{P}_{2g+1}}\sum_{\substack{P\in\mathcal{P}_{2g+1}}} |M(n; \Lambda_k \chi_P)|^2 \left(1+O\left(\frac{1}{q}\right)\right).
 \end{equation}
By combining with Lemma \ref{lem:M-vM}, we get 
\begin{equation}\label{eq:var}
 \mathrm{Var}(\mathcal{S}^S_{\Lambda_k,n})
 = \frac{q^n}{4\# \mathcal{P}_{2g+1}}\sum_{\substack{P\in\mathcal{P}_{2g+1}}}\left|\sum_{\substack{j_1+\cdots+j_k=n\\1\leq j_1,\dots, j_k}} \mathrm{Tr}(\Theta_{C_P}^{j_1})\cdots \mathrm{Tr}(\Theta_{C_P}^{j_k})\right|^2 \left(1+O\left(\frac{1}{q}\right)\right).
\end{equation}
 
 Applying Katz's Theorem \ref{thm:Katz},  we obtain the statement of Theorem \ref{thm:vM-S-s}. 
 \begin{thm} Let $k\leq n$. As $q\rightarrow \infty$, 
\[\mathrm{Var}(\mathcal{S}^S_{\Lambda_k,n})
\sim \frac{q^n}{4}\int_{\mathrm{Sp}(2g)}\left|\sum_{\substack{j_1+\cdots+j_k=n\\1\leq j_1,\dots, j_k}} \mathrm{Tr}(U^{j_1})\cdots \mathrm{Tr}(U^{j_k})\right|^2 dU.\]
 \end{thm}
 
 \subsection{Average of the von Mangoldt convolution function over short arcs on the unit circle}

Now we consider the question of the distribution of convolutions of the von Mangoldt function over sectors of the unit circle. Our goal is to study the mean and variance of the following sum:
\[\mathcal{N}^S_{\Lambda_\ell,k,\nu}(v)=\sum_{\substack{f \in \mathcal{M}_\nu\\ f(0)\not =0\\U(f)\in \mathrm{Sect}(v,k)}} \Lambda_\ell(f).\]
Rudnick and Waxman \cite{Rudnick-Waxman} study this question for the case $\ell=1$ as an intermediate step to estimating \eqref{eq:N-RW}.

As before, let
\[M_0(\nu;\Lambda_\ell\Xi)= \sum_{\substack{f\in \mathcal{M}_{\nu}\\f(0)\not = 0}} \Lambda_\ell(f) \Xi(f) .\]

  Lemma \ref{lem:genML} becomes the following statement. 
\begin{lem}\label{lem:M-vM-RW}
We have, for $\ell \leq \nu$, 
 \begin{align*}
 M_0(\nu ;\Lambda_\ell\Xi)
=& (-1)^{\ell}q^{\nu/2} \sum_{\substack{j_1+\cdots+j_\ell=\nu\\1\leq j_1,\dots,j_\ell }}  \mathrm{Tr}(\Theta_\Xi^{j_1})\cdots \mathrm{Tr}(\Theta_\Xi^{j_\ell})  +O\left(q^\frac{\nu-1}{2}\right),
  \end{align*}
  and $M_0(\nu ;\Lambda_\ell\Xi)=0$ for $1\leq n<\ell$. 
\end{lem}

We start our analysis by looking at  the mean value averaging over all the directions of $v\in \mathbb{S}_k^1$. 
\begin{lem}\label{lem:averageNvMS} We have 
\begin{align*}
 \langle \mathcal{N}^S_{\Lambda_\ell,k,\nu}\rangle =&\frac{1}{q^\kappa} \sum_{v \in \mathbb{S}_k^1} \mathcal{N}^S_{\Lambda_\ell,k,\nu}(v)
 =\frac{1}{q^\kappa} \sum_{\substack{f \in \mathcal{M}_\nu\\ f(0)\not =0}} \Lambda_\ell(f)=q^{\nu-\kappa} \binom{\nu-1}{\ell-1} +O\left(q^{\nu-\kappa-1} \right). 
 \end{align*}
\end{lem} 
 \begin{proof}
We have computed this term before in \eqref{eq:lambdapnd}. Since this case is particularly simple, we redo it here. The generating function is 
\begin{align}\label{eq:sumfnot0exact}
\sum_{\substack{f \in \mathcal{M}\\ f(0)\not =0}} \Lambda_\ell(f)u^\nu=& \left(\frac{u\mathcal{Z}'(u)}{\mathcal{Z}(u)} -\frac{u}{1-u}\right)^\ell= \left(\frac{qu}{1-qu} -\frac{u}{1-u}\right)^\ell= \left(\frac{u(q-1)}{(1-u)(1-qu)}\right)^\ell\nonumber\\
=& (q-1)^\ell\sum_{\nu=\ell}^\infty \sum_{m=0}^\infty \binom{-\ell}{m} \binom{-\ell}{\nu-\ell-m} (-1)^{n-\ell} q^{\nu-\ell-m} u^\nu.
\end{align}
Comparing coefficients, we get
\begin{align}\label{eq:sumfnot0}
\sum_{\substack{f \in \mathcal{M}_\nu\\ f(0)\not =0}} \Lambda_\ell(f)=& q^{\nu} \binom{-\ell}{\nu-\ell} (-1)^{\nu-\ell}  +O(q^{\nu-1})=q^{\nu} \binom{\nu-1}{\ell-1}  +O(q^{\nu-1}).
\end{align}
\end{proof}

By orthogonality of super-even characters,
\[\mathcal{N}^S_{\Lambda_\ell,k,\nu}(v)=\sum_{\substack{f \in \mathcal{M}_\nu\\ f(0)\not =0\\U(f)\in \mathrm{Sect}(u,k)}} \Lambda_\ell(f)= \frac{1}{q^{\kappa}}\sum_{\Xi \text{ super-even} \pmod{S^k}} \overline{\Xi(u)} \sum_{\substack{f \in \mathcal{M}_\nu\\ f(0)\not =0}} \Lambda_\ell(f)\Xi(f).\]
The contribution from the trivial character $\Xi_0$ is precisely \[\frac{1}{q^{\kappa}}\sum_{\substack{f \in \mathcal{M}_\nu\\ f(0)\not =0}} \Lambda_\ell(f)= \langle \mathcal{N}^S_{\Lambda_\ell,k,\nu}\rangle.\] 
Thus
\begin{align*}
\mathcal{N}^S_{\Lambda_\ell,k,\nu}(v)- \langle \mathcal{N}^S_{\Lambda_\ell,k,\nu}\rangle=&\frac{1}{q^{\kappa}}\sum_{\substack{\Xi \text{ super-even} \pmod{S^k}\\ \Xi\not = \Xi_0}} \overline{\Xi(v)} \sum_{\substack{f \in \mathcal{M}_\nu\\ f(0)\not =0}} \Lambda_\ell(f)\Xi(f)\nonumber \\
=&\frac{1}{q^{\kappa}}\sum_{\substack{\Xi \text{ super-even} \pmod{S^k}\\ \Xi\not = \Xi_0}}  \overline{\Xi(v)} M_0(\nu; \Lambda_\ell \Xi). \label{eq:sumM0}
\end{align*}

We consider the variance
\begin{align*}
 \mathrm{Var}(\mathcal{N}^S_{\Lambda_\ell,k,\nu})=& \frac{1}{q^\kappa}\sum_{u \in \mathbb{S}_k^1} \left|\mathcal{N}^S_{d_\ell,k,\nu}(u)- \langle \mathcal{N}^S_{\Lambda_\ell,k,\nu}\rangle\right|^2.
\end{align*}

By applying the orthogonality relations \eqref{eq:orthogonalitysupereven} we obtain
 \begin{align*}
 \mathrm{Var}(\mathcal{N}^S_{\Lambda_\ell,k,\nu})=& \frac{1}{q^\kappa}\sum_{u \in \mathbb{S}_k^1}
 \frac{1}{q^{2\kappa}}\sum_{\substack{\Xi_1, \Xi_2 \text{ super-even} \pmod{S^k}\\ \Xi_1, \Xi_2\not = \Xi_0}}  \overline{\Xi_1(u)} M_0(\nu; \Lambda_\ell \Xi_1) \Xi_2(u) \overline{M_0(\nu; \Lambda_\ell \Xi_2)}\\
 =&
 \frac{1}{q^{2\kappa}}\sum_{\substack{\Xi_1, \Xi_2 \text{ super-even} \pmod{S^k}\\ \Xi_1, \Xi_2\not = \Xi_0}} M_0(\nu; \Lambda_\ell \Xi_1) \overline{M_0(\nu; \Lambda_\ell \Xi_2)} \frac{1}{q^\kappa}\sum_{u \in \mathbb{S}_k^1} \overline{\Xi_1(u)}\Xi_2(u)  \\
 =&
 \frac{1}{q^{2\kappa}}\sum_{\substack{\Xi\text{ super-even} \pmod{S^k}\\ \Xi\not = \Xi_0}} |M_0(\nu; \Lambda_\ell \Xi)|^2.
\end{align*}

Finally, we apply Lemma \ref{lem:M-vM-RW} in order to  obtain, 
\begin{equation}
\mathrm{Var}(\mathcal{N}^S_{\Lambda_\ell,k,\nu})
 = \frac{q^{\nu}}{q^{2\kappa}}\sum_{\substack{\Xi\text{ super-even} \pmod{S^k}\\ \Xi\not = \Xi_0}} \left|\sum_{
 \substack{j_1+\cdots+j_\ell=\nu\\1\leq j_1,\dots,j_\ell }} 
 \mathrm{Tr}(\Theta_\Xi^{j_1})\cdots \mathrm{Tr}(\Theta_\Xi^{j_\ell}) \right|^2\left(1+O\left(\frac{1}{\sqrt{q}}\right)\right).
\end{equation}

We have now all the elements to prove Theorem \ref{thm:vM-S-n}.
\begin{thm} Let $\ell \leq \nu$.  As $q \rightarrow \infty$,
\[\mathrm{Var}(\mathcal{N}^S_{\Lambda_\ell,k,\nu})\sim\frac{q^\nu}{q^{\kappa}}\int_{\mathrm{Sp}(2\kappa-2)}\left| \sum_{\substack{j_1+\cdots+j_\ell=\nu\\1\leq j_1,\dots,j_\ell }}\mathrm{Tr}(U^{j_1})\cdots \mathrm{Tr}(U^{j_\ell}) \right|^2 dU.\]
\end{thm}

\begin{proof}
As in the proof of Theorem \ref{thm:sector-text}, we separate the characters according to their Swan conductor, which is an odd integer $d(\Xi)<k$ with maximal value $2\kappa-1$. The characters with maximal conductor are primitive, and the contribution from the others is negligible. So, we consider the sum over the primitive characters. By applying the result by Katz \cite[Theorem 5.1]{Katz}, we obtain the desired result. 
\end{proof}

\subsection{Relationship with Random Matrix Theory} 
In this section we discuss  the computation of the integral
\begin{equation}\label{eq:int-symplectic}
I_{\Lambda_k}^S(n;N):=\int_{\mathrm{Sp}(2N)} \sum_{\substack{j_1+\cdots+j_k=n\\1\leq j_1,\dots,j_k }}\mathrm{Tr}(U^{j_1})\cdots \mathrm{Tr}(U^{j_k}) dU.
\end{equation}
Note that $I_{\Lambda_k}^S(n;N)$ measures the discrepancy of the objects that we discuss above, rather than the variance. However, from the point of view of random matrix theory, this  is a more natural object to consider. The methods that we used to prove Theorem \ref{thm:vM-S-s}
apply to obtain a similar statement for the discrepancy:
\begin{thm}\label{thm:discrepancy}
 Assume that $k\leq n$. As $q\rightarrow \infty$, 
\begin{align*}\frac{1}{\# \mathcal{P}_{2g+1}}\sum_{P\in \mathcal{P}_{2g+1}} &\left(\mathcal{S}^S_{\Lambda_k,n}(P) -\frac{1}{2}\sum_{\substack{f\in \mathcal{M}_n \\P\nmid f}} \Lambda_k(f) \right)
\\
&\sim \frac{(-1)^kq^{n/2}}{2}\int_{\mathrm{Sp}(2g)}\sum_{\substack{j_1+\cdots+j_k=n\\1\leq j_1,\dots, j_k}} \mathrm{Tr}(U^{j_1})\cdots \mathrm{Tr}(U^{j_k}) dU.
\end{align*}
\end{thm}
\begin{rem} \label{rem:the mean is meaner}
Note that $\frac{1}{2}\sum_{\substack{f\in \mathcal{M}_n \\P\nmid f}} \Lambda_k(f)$ denotes the main term of $\mathcal{S}^S_{\Lambda_k,n}(P)$, so Theorem \ref{thm:discrepancy} picks up on a contribution from the second-order terms. In the case of 
$\mathcal N_{\Lambda_\ell,k,\nu}^S(v)$, $\langle \mathcal N_{\Lambda_\ell,k,\nu}^S \rangle$ is the limiting mean value rather than the main term, so the discrepancy above goes to zero as $q \rightarrow \infty$.
\end{rem}

Back to integral \eqref{eq:int-symplectic},
Diaconis and Shahshahani \cite{Diaconis-Shahshahani} prove the following result (see also \cite{Diaconis-Evans}). 

 \begin{thm} \label{thm:DS-symplectic}\cite[Theorem 6]{Diaconis-Shahshahani} Let $U$ be Haar distributed on $\mathrm{Sp}(2N)$. Let ${\bf a}=(a_1,\dots, a_h)$ with $a_i \in \Z_{\geq0}$. Then for $N\geq \sum_{j=1}^ha_jj$, 
\begin{equation}\label{eq:DS-symplectic}
\mathbb{E}\left(\prod_{j=1}^h (\mathrm{Tr} (U^j))^{a_j}  \right)=\prod_{j=1}^h (-1)^{(j-1)a_j}g_j( a_j),
\end{equation}
where if $j$ is odd,
\begin{equation}\label{eq:gjajodd}
g_j(a)=\begin{cases} 0 & a \mbox{ odd}, \\ j^{a/2} (a-1)!! & a \mbox{ even},
         \end{cases}
         \end{equation}
and if $j$ is even, 
\begin{equation}\label{eq:gjajeven}
g_j(a)=\sum_{\ell=0}^{\lfloor a/2\rfloor} \binom{a}{2\ell} j^{\ell} (2\ell-1)!!
\end{equation}
\end{thm}

The above result is limited by the condition $N\geq \sum_{j=1}^ha_jj$. 
Keating and Odgers \cite{Keating-Odgers} give a precise statement over a larger interval for the product of one and two traces. (See also Hughes and Rudnick \cite{Hughes-Rudnick}.)
\begin{lem}\label{lem:KO}\cite[Lemma 2]{Keating-Odgers}
Over $\mathrm{Sp}(2N)$, we have
\[\mathbb{E}(\mathrm{Tr}(U^j)) =\begin{cases}
                                  2N & j =0, \\
                                   -\eta_j & 1\leq |j|\leq 2N, \\
                                   0 & \mbox{ otherwise,}
                                \end{cases}\]
                                where we recall that $\eta_j$ is given by equation \eqref{eq:eta}. 
For $0<j_1\leq j_2$, 
\begin{align*}
\mathbb{E}(\mathrm{Tr}(U^{j_1})\mathrm{Tr}(U^{j_2}))=& j_1 (\mbox{if }1\leq j_1=j_2\leq 2N)\\
&+2N (\mbox{if } j_1=j_2>2N)\\
&-1(\mbox{if } N+1\leq j_1=j_2\leq 2N)\\
&+\eta_{j_1}\eta_{j_2} (\mbox{if } 1\leq j_1\leq j_2 \leq 2N)\\
&-1 (\mbox{if } j_1=n-m, j_2=n+m, 1\leq m\leq N, n\geq N+1).
 \end{align*}
 
\end{lem}

As observed in Remark \ref{rem:the mean is meaner}, there are cases in which it would be more natural to consider the integral 
\begin{equation}\label{eq:int-symplectic-square}
I_{\Lambda_k,2}^S(n;N):=\int_{\mathrm{Sp}(2N)} \left|\sum_{\substack{j_1+\cdots+j_k=n\\1\leq j_1,\dots,j_k }}\mathrm{Tr}(U^{j_1})\cdots \mathrm{Tr}(U^{j_k})\right|^2 dU.
\end{equation}

In fact, Lemma \ref{lem:KO} automatically gives a formula for $I_{\Lambda_1,2}^S(n;N)$ (also stated in \cite{Rudnick-Waxman}), below.

\begin{cor}
If $k = 1$, then $I_{\Lambda_k,2}^S(n;N)$ is given by
\[\int_{\mathrm{Sp}(2N)} \left|\mathrm{Tr}(U^{n})\right|^2 dU = \begin{cases}n+\eta_n & 1\leq n \leq N,\\
n-1+\eta_n &N+1\leq n \leq 2N,\\
2N & 2N< n.
             \end{cases}\] 
As $N \to \infty$, $I^S_{\Lambda_k,2}(n;N)$ is asymptotic to
\[I^S_{\Lambda_1,2}(n;N) \sim \gamma_{\Lambda_1,2}(c)N, \]
where $c = n/2N$ and
\[\gamma_{\Lambda_1,2}^S(c)=\begin{cases} 
                 c & 0\leq c \leq 2,\\
                2 & 2\leq c.
                \end{cases}
\]
\end{cor}

We study a more general setting for $I_{\Lambda_k}^S(n;N)$.  For a matrix $U\in \mathrm{Sp}(2N)$, consider
\[\mathcal{L}_U(s):=\det(I-sU)=\prod_{n=1}^{2N} (1-se^{-i\theta_n}).\]
This $L$-function satisfies the functional equation 
\[\mathcal{L}_U(s)=\det U^*s^N \mathcal{L}_{U^*} (1/s).\]

Mason and Snaith \cite{Mason-Snaith} prove the following result. 
 \begin{thm}\label{thm:MS}\cite[Theorem 3.3]{Mason-Snaith} Let $A=\{\alpha_1,\dots,\alpha_k\}$, such that $\re(\alpha_j)>0$, then 
 \begin{equation}\label{eq:int-generating-symplectic}
 \int_{\mathrm{Sp}(2N)} \prod_{\alpha\in A} (-e^{-\alpha})\frac{\mathcal{L}_U'}{\mathcal{L}_U} (e^{-\alpha}) dU
 \end{equation}
 is equal to 
\begin{equation}\label{eq:int-generating-symplectic-answer}\sum_{S\subseteq A} e^{-2N\sum_{\hat{\alpha}\in S}\hat{\alpha}}(-1)^{|S|} \sqrt{\frac{Z(S,S)Z(S^-,S^-)Y(S^-)}{Y(S)Z^\dagger (S^-,S)^2}} \sum_{\substack{A-S=W_1+\cdots+W_R\\|W_r|\leq 2}}
\prod_{r=1}^R H_{S}(W_r).
\end{equation}
Here 
\[S^-=\{-\hat{\alpha}, \hat{\alpha}\in S\},\]
\[Y(A)=\prod_{\alpha\in A} z(2\alpha),\quad \mbox{ and }\quad Z(A,B)=\prod_{\substack{\alpha \in A\\\beta\in B}} z(\alpha+\beta),\]
where $z(\alpha)=\frac{1}{1-e^{-\alpha}}$, and the dagger imposes the additional restriction that a factor $z(\alpha)$ is omitted when the argument is zero.

Finally, the sum over $W_r$ is a sum over all the different set partitions of $A-S$ and 
\[H_{S}(W)=\begin{cases}\left(\sum_{\hat{\alpha}\in S} \frac{z'}{z}(\alpha-\hat{\alpha}) - \frac{z'}{z}(\alpha+\hat{\alpha})\right)+ \frac{z'}{z}(2\alpha) & W=\{\alpha\}\subseteq A-S,\\
\left(\frac{z'}{z}\right)'(\alpha+\beta) & W=\{\alpha, \beta\}\subseteq  A-S,\\
1 & W=\varnothing.
              \end{cases}\]
\end{thm}

Note that $\frac{z'}{z}(\alpha)=-\frac{e^{-\alpha}}{1-e^{-\alpha}}$, $\left(\frac{z'}{z}\right)'(\alpha)=\frac{e^{-\alpha}}{(1-e^{-\alpha})^2}$.

We separate the sum over $S\subset A$ according to $|S|$. We refer to the terms arising from $|S|=\ell$ as $\ell$-swap terms.  From now on, we will set $x_i:=e^{-\alpha_i}$. The integral then gives a generating series for  \eqref{eq:int-symplectic}. Since the goal is to have a fixed total sum of exponents, we can set all $x_i=x$ (provided that there are no poles) and search for the coefficient of $x^n$. 
After doing this, the $\ell$-swap terms are coefficients of $x^{2\ell N}$. Therefore, they only appear for $n\geq 2\ell N$. For example, it is immediate to see from formula \eqref{eq:int-generating-symplectic-answer} that when $n<2N$, we only get the $0$-swap terms. In fact, this is the case when $n<2N+2$. The $0$-swap terms in this should give the same result as Theorem  \ref{thm:DS-symplectic}.

\subsubsection{$0$-swap terms} The 0-swap terms arise from taking $S$ empty. The $0$-swap terms of the integral in \eqref {eq:int-generating-symplectic} are then given by terms of the form
\[\sum_{\substack{A=W_1+\cdots+W_R\\|W_r|\leq 2}}
\prod_{r=1}^R H_{\varnothing}(W_r).\]
Each factor $H_{\varnothing}(W)$ is either of the form $\frac{-x^2}{1-x^2}$ when $|W|=1$ or of the form  $\frac{x_ix_j}{(1-x_ix_j)^2}$ when $|W|=2$.

First we will see how this compares with Theorem 6.9. Let $s$ be the number of subsets $W_r$ with cardinality $2$, so that $s + R = k$. For each $W_r$ of cardinality $2$, let $P_{2,r} = \prod_{\alpha_i \in W_r} x_i$, and for each $W_r$ of cardinality $1$, let $P_{1,r} = x_i$, where $\alpha_i \in W_r$. Summing over $s$, we get that the final generating function is given by
\begin{equation}\label{eq:gen-DS-symplectic}
(-1)^k\sum_{s \le \lfloor k/2 \rfloor} \sum_{\substack{A = W_1 + \cdots + W_{k-s} \\ 1 \le |W_r| \le 2}} \prod_{|W_r| = 2} \frac{P_{2,r}}{(1-P_{2,r})^2} \prod_{|W_r| = 1} \frac{P_{1,r}^2}{1-P_{1,r}^2}.
\end{equation}
Expanding the products above gives
\begin{align*}
(-1)^k \sum_{s \le \lfloor k/2 \rfloor} \sum_{\substack{A = W_1 + \cdots + W_{k-s} \\ 1 \le |W_r| \le 2}} \sum_{\substack{m_r \ge 1 \\ |W_r| = 2}} m_r P_{2,r}^{m_r} \sum_{\substack{m_r \ge 1 \\ |W_r| = 1}} P_{1,r}^{2m_r}.
\end{align*}

In order to recover Theorem \ref{thm:DS-symplectic}, we focus on the coefficient of 
\[X_{{\bf a}} := x_1 \cdots x_{a_1} x_{a_1 + 1}^2 \cdots x_{a_1 + a_2}^2 x_{a_1 + a_2 +1}^3\cdots \]
In other words, we have a monomial where the first $a_1$ variables have exponent $1$, then the next $a_2$ variables have exponent $2$, and so on. We have that $a_1 + \cdots + a_k = k$, and if any $a_j$ is $0$, then there are no variables with exponent $j$. We want to check that the coefficient of $X_{{\bf a}}$ is the same as the one predicted by formula \eqref{eq:DS-symplectic}. Note that the ordering of the variables $x_i$ does not affect the coefficient either in our calculation or in formula \eqref{eq:DS-symplectic}, which is why it suffices to check the coefficient when the variables are ordered according to their exponent, as above.

We first focus on those variables with a fixed exponent $j$. If $j$ is odd, then the only contributions to the coefficient of $X_{{\bf a}}$ coming from the factors $x_{a_{j-1}+1}^j \cdots x_{a_j}^j$ arise from terms of the form $m_rP_{2,r}^{m_r}$, where $|W_r| = 2$ and $\alpha_j \in W_r$. Since $W_r$ also contains a second variable, the total number of variables with exponent $j$ in $X_{\bf a}$ must be even, or in other words $a_j$ is even. In this case, we have
\[\frac{a_j!}{2^{a_j/2}\left(\frac{a_j}{2}\right)!} = (a_j-1)!!\]
ways of partitioning the indices of the variables $x_{a_{j-1}+1}, \dots, x_{a_j}$ into pairs; each partition has a coefficient $j^{a_j/2}$. This is precisely formula \eqref{eq:gjajodd}.

Meanwhile if $j$ is even, then the contributions to the coefficient of $x_{a_{j-1}+1}^j \cdots x_{a_j}^j$ arise from cases where, for some $\ell \le a_j/2$, $\ell$ pairs of the variables appear in two-element subsets $W_r$, and the remaining $a_j - 2\ell$ variables appear in single-element subsets. The coefficient coming from the $\ell$ pairs of variables is $(2\ell-1)!!j^\ell$, as before, whereas the single-element subsets do not change the coefficient of $X_{\bf a}$. For each $\ell$, there are $\binom{a_j}{2\ell}$ ways of choosing which variables lie in two-element subsets, so the coefficient is given by
\[\sum_{\ell = 0}^{\lfloor a_j/2\rfloor} \binom{a_j}{2\ell} j^\ell (2\ell-1)!!,\]
which is precisely formula \eqref{eq:gjajeven}.

Proceeding inductively on $j$, we must multiply all values of $g_j(a_j)$ in order to obtain the coefficient of $X_{{\bf a}}$. Note that $g_j(0)$ is always $1$. The coefficient is identical up to the sign; in particular, we must verify that the coefficient of $X_{{\bf a}}$ in formula \eqref{eq:DS-symplectic} has sign $(-1)^k$. But the sign in this formula is
\[(-1)^{\sum_{j=1}^h ja_j-a_j} = (-1)^k (-1)^{\sum_{j=1}^h ja_j} = (-1)^k,\]
since $a_1 + \cdots + a_h = k$, and since for a nonzero coefficient, $j$ and $a_j$ cannot both be odd. Thus we have recovered the result of Theorem \ref{thm:DS-symplectic} from Theorem \ref{thm:MS}.

For the question of computing the $0$-swap terms in integral \eqref{eq:int-symplectic}, we can let all the variables in \eqref{eq:gen-DS-symplectic} be the same, as previously discussed. We get
  \begin{align*}&(-1)^k\sum_{s\leq \lfloor k/2\rfloor}\binom{k}{2s} \frac{(2s)!}{2^s s!}  \left(\frac{x^2}{(1-x^2)^2}\right)^s \left(\frac{x^2}{1-x^2}\right)^{k-2s}\\
  =&\left(\frac{-1}{1-x^2}\right)^k \sum_{s\leq \lfloor k/2\rfloor}\binom{k}{2s} (2s-1)!! x^{2k-2s},
 \end{align*}
where the coefficient of each term in the first sum is counting the number of ways of choosing $s$ sets $W$ of cardinality 2 and $k-2s$ sets $W$ of cardinality $1$. 

This gives
 \begin{align*}
 &(-1)^k\sum_{j\geq0} \binom{k+j-1}{j}x^{2j} \sum_{s\leq \lfloor k/2\rfloor}\binom{k}{2s} (2s-1)!! x^{2k-2s}\\
=&(-1)^k\sum_{\substack{n\geq k}} x^{2n} \sum_{s\leq \lfloor k/2\rfloor}\binom{n+s-1}{n+s-k}  \binom{k}{2s} (2s-1)!! \\
=& (-1)^kk\sum_{\substack{n\geq k}} x^{2n} \sum_{s\leq \lfloor k/2\rfloor}\frac{(n+s-1)!}{2^s s! (n+s-k)!(k-2s)!}
 \end{align*}
 for the final contribution of the $0$-swaps terms.

\subsubsection{$1$-swap terms} The $1$-swap terms are more difficult to compute than the $0$-swap terms. In this section we do this computation for $k=2$. First notice that the 1-swap terms are given by 
\begin{align*}
&-\frac{x_1^{2N+2}}{x_1^2-1}\left(-\frac{x_2}{x_1-x_2}+\frac{x_1x_2}{1-x_1x_2}-\frac{x_2^2}{1-x_2^2}\right)-\frac{x_2^{2N+2}}{x_2^2-1}\left(-\frac{x_1}{x_2-x_1}+\frac{x_1x_2}{1-x_1x_2}-\frac{x_1^2}{1-x_1^2}\right)\\
=& x_1x_2\frac{x_1^{2N+1}(x_1x_2^3-2x_2^2+x_1x_2-x_1^2+1)-x_2^{2N+1} (x_2x_1^3-2x_1^2+x_1x_2-x_2^2+1)}{(x_1^2-1)(x_2^2-1)(x_1-x_2)(x_1x_2-1)}\\
\end{align*}
By factoring $(x_1-x_2)$ from the numerator and taking the limit as $x_2 \to x_1$, we obtain 
\begin{align*}
&-x^2\frac{(2N-1)x^{2N+4}-4Nx^{2N+2}+(2N+1)x^{2N}}{(1-x^2)^3}\\=&-x^{2N+2}\frac{2N(1-x^2)^2+(1-x^4)}{(1-x^2)^3}\\
=&-x^{2N+2}\left(\frac{2N}{1-x^2}+\frac{(1+x^2)}{(1-x^2)^2}\right)\\
=&-2Nx^{2N+2}\sum_{j\geq 0} x^{2j}-x^{2N}\sum_{j\geq 1} jx^{2j} -x^{2N+2}\sum_{j\geq 0} jx^{2j}\\
=&-x^{2N}\sum_{j\geq 0} 2Nx^{2j+2}-x^{2N}\sum_{j\geq 1} jx^{2j} -x^{2N}\sum_{j\geq 0} jx^{2j+2}\\
=&-\sum_{j\geq 1} (2N+2j-1)x^{2N+2j}.
\end{align*}
\subsubsection{$2$-swap terms} For $k=2$, the $2$-swap terms are 
\begin{align*}
x_1^{2N+1}&x_2^{2N+1} \frac{x_1x_2(x_1-x_2)^2}{(1-x_1x_2)^2(1-x_1^2)(1-x_2^2)}\\
=&x_1^{2N+1}x_2^{2N+1}(x_1^2-2x_1x_2+x_2^2)\sum_{j_0\geq0} x_1^{2j_0}\sum_{h_0\geq 0} x_2^{2h_0}  \sum_{j_1\geq 1} x_1^j x_2^j,
\end{align*}
which yields 0 when $x_2\rightarrow x_1$.

When $k = 2$, the following corollary is a result of the above computations.

\begin{cor}
For $k = 2$, the integral $I_{\Lambda_2}^S(n;N)$ is given by
\[\int_{\mathrm{Sp}(2N)} \sum_{\substack{j_1 + j_2 = n \\ 1 \le j_1, j_2}} \mathrm{Tr}(U^{j_1})\mathrm{Tr}(U^{j_2}) dU =\begin{cases}\left(n-1\right)\eta_n & 2\leq n \leq 2N+1,\\
0 &2N+2\leq n.
             \end{cases}\]
As $N \to \infty$ and for $n$ even, $I_{\Lambda_2}^S(n;N)$ is asymptotic to
\[I_{\Lambda_2}^S(n;N) \sim \gamma_{\Lambda_2}^S(c)N,\]
where $c = n/2N$ and
\[\gamma_{\Lambda_2}^S(c)=\begin{cases} 
                 c & 0\leq c \leq 1,\\
                0 & 1\leq c.
                \end{cases}
\]
\end{cor}
\begin{proof}
By gathering together the contributions from the $0$-swap, $1$-swap, and $2$-swap terms in the case $k=2$, we get
\[\sum_{n\geq 2}I_{\Lambda_2}^S(n;N)x^n=2\sum_{\substack{m\geq 2}} \left(m-\frac{1}{2} \right) x^{2m} -\sum_{j\geq 1} (2N+2j-1)x^{2N+2j}.\]
\end{proof}
Note that we could have recovered this formula from Lemma \ref{lem:KO} as well.

\section{Unitary averages of the von Mangoldt convolution} \label{sec:vonMangoldtunitary}

 The goal of this section is to recover the results of \cite{KR3} for the convolution of the von Mangoldt function. These results were obtained by Keating and Rudnick \cite{Keating-Rudnick} for the von Mangoldt function itself; here we consider $k$-fold convolutions of the von Mangoldt function, as in the previous section. We will follow several of the ideas and arguments from \cite{Keating-Rudnick}. 
 
 \subsection{Average of the von Mangoldt convolution function over the short intervals}

We start by recalling the notation of short intervals in $\F_q[T]$. Let $A\in \mathcal{M}_n$ and let $0\leq h \leq n-2$. Then let \[I(A;h)=\{f\in \mathcal{M}\,:\, |f-A|\leq q^h\}.\]

We are interested in studying the following sum, which is the von Mangoldt convolution analogue to the problem in Section 2 of  
 \cite{Keating-Rudnick}:
\[\mathcal{N}^U_{0,\Lambda_k}(A;h)=\sum_{\substack{f\in I(A;h)\\f(0)\not = 0}}\Lambda_k(f).\]

 First we consider
 \[M_0(n;\Lambda_k\chi)= \sum_{\substack{f\in \mathcal{M}_n\\f(0)\not = 0}} \Lambda_k(f) \chi(f) .\]
 
 Lemma \ref{lem:genML} becomes the following statement. \begin{lem} Let $\chi$ be even.  \label{lem:M0-vM-interval}
 We have, for $k\leq n$, 
 \begin{align*}
 M_0(n;\Lambda_k\chi)
=& (-1)^k q^{n/2}\sum_{\substack{j_1+\cdots+j_k=n\\1\leq j_1,\dots,j_k }}  \mathrm{Tr}(\Theta_\chi^{j_1})\cdots \mathrm{Tr}(\Theta_\chi^{j_k})  +O\left(q^\frac{n-1}{2}\right),
  \end{align*}
  and $M_0(n ;\Lambda_k\chi)=0$ for $1\leq n<k$. 
\end{lem}

Now we compute the mean.
\begin{lem} We have
\[\langle \mathcal{N}^U_{0,\Lambda_k}\rangle= \frac{1}{q^n} \sum_{A\in \mathcal{M}_n}\mathcal{N}^U_{0,\Lambda_k}(A;h)=\frac{q^{h+1}}{q^n} \sum_{\substack{f\in \mathcal{M}_n\\f(0)\not = 0}}\Lambda_k(f)=q^{h+1}\binom{n-1}{k-1}+O\left(q^h \right).\]
\end{lem}
\begin{proof}
See Lemma \ref{lem:averageNvMS}.
\end{proof}

We remark that Keating and Rudnick have a more precise expression in \cite{Keating-Rudnick}. Indeed, by \eqref{eq:sumfnot0exact}, 
\[\sum_{\substack{f\in \mathcal{M}_n\\f(0)\not = 0}}\Lambda_k(f) =(q-1)^k\sum_{m=0}^\infty \binom{-k}{m} \binom{-k}{n-k-m} (-1)^{n-k} q^{n-k-m}.\]
By setting $k=1$, we get
\[\sum_{\substack{f\in \mathcal{M}_n\\f(0)\not = 0}}\Lambda(f) =q^n-1,\]
and this results in 
\[\langle \mathcal{N}^U_{0,\Lambda}\rangle=q^{h+1}\left(1-\frac{1}{q^n}\right),\]
as in \cite{Keating-Rudnick}.

Our next goal is to find the variance
\begin{align*}
 \mathrm{Var}(\mathcal{N}^U_{0,\Lambda_k})=& \frac{1}{q^n}\sum_{A \in \mathcal{M}_n}  \left|\mathcal{N}^U_{0,\Lambda_k}(A;h)- \langle \mathcal{N}^U_{0,\Lambda_k}\rangle\right|^2.
\end{align*}

Next we will follow several ideas of  \cite{Keating-Rudnick}. Define
\[\tilde{\Psi}_k(n;Q,A)=\sum_{\substack{\deg(f)=n\\f\equiv A \pmod{Q}}} \Lambda_k(f),\]
where the sum takes places over all the polynomials of degree $n$, not necessarily monic. 

For a $f \in \F_q[T]$ with $f(0)\not = 0$, we consider the following involution:
\[f^*(T)=T^{\deg(f)}f\left( \frac{1}{T}\right).\]
Then we have that $(fg)^*=f^*g^*$ and $\Lambda_k(f^*)=\Lambda_k(f)$, with a proof very similar to \cite[Lemma 4.1]{Keating-Rudnick}.

Notice that, for a fixed $h$, every $f\in \mathcal{M}_n$ with $n\geq h+1$ can be written uniquely as 
\[f=T^{h+1}B+g, \quad B\in \mathcal{M}_{n-(h+1)}, \quad \deg(g) \leq h.\]
We can then decompose $\mathcal{M}_n$ as 
\begin{equation}\label{eq:M}\mathcal{M}_n=\bigsqcup_{B \in  \mathcal{M}_{n-(h+1)}} I(T^{h+1}B;h).
\end{equation}
The involution $*$ gives a bijection 
\begin{eqnarray*}
*:\mathcal{M}_{n-(h+1)} &\rightarrow& \{\deg(B^*)\leq n-h-1 \, : \, B^*(0)=1\}\\
B&\mapsto& B^*. 
\end{eqnarray*}

\begin{lem}\label{lem:N-Psi}
 Let $B\in \F_q[T]$ such that $\deg(B)=n-h-1$. Then 
 \[\mathcal{N}^U_{0,\Lambda_k}(T^{h+1}B;h)=\tilde{\Psi}_k(n;T^{n-h},B^*).\]
\end{lem}
\begin{proof}
 This is an extension of \cite[Lemma 4.2]{Keating-Rudnick}. It is proven that 
 \[f\in I(T^{h+1}B;h)\Leftrightarrow f^*\equiv B^* \pmod{T^{n-h}}.\]
The result then follows  because if $f$ runs over $I(T^{h+1}B; h)$ such that $f(0)\not = 0$, then $f^*$ runs over all polynomials of degree exactly $n$ satisfying $f^*\equiv B^*\pmod{T^{n-h}}$, and for these, $\Lambda_k(f^*)=\Lambda_k(f)$. 
 \end{proof}

 By the orthogonality relations for Dirichlet characters, 
\begin{equation}\label{eq:Psi}
\tilde{\Psi}_k(n;T^{n-h},B^*)=\frac{1}{\Phi(T^{n-h})} \sum_{\chi \pmod{T^{n-h}}}\overline{\chi}(B^*)\sum_{\deg(f^*)=n} \Lambda_k(f^*) \chi(f^*).\end{equation}
Notice that only even characters contribute, since $\Lambda_k(cf)=\Lambda_k(f)$ for $c \in \F_q^\times$, and therefore, an odd character produces a factor of the form $\sum_{c\in \F_q^\times}\chi(c)=0$ in the inner sum. 
When the character is even, it contributes with a term of the form 
\[\overline{\chi}(B^*)\frac{q-1}{\Phi(T^{n-h})} \sum_{\substack{f\in \mathcal{M}_n\\f(0)\not = 0}} \Lambda_k(f) \chi(f)=\frac{\overline{\chi}(B^*)}{q^{n-h-1}}M_0(n;\Lambda_k\chi).\]
The number of even characters modulo $T^{n-h}$ is $\frac{\Phi(T^{n-h})}{q-1}=q^{n-h-1}$. 
The trivial character $\chi_0$ contributes the term 
\[\frac{q-1}{\Phi(T^{n-h})} \sum_{\substack{f\in \mathcal{M}_n\\f(0)\not = 0}} \Lambda_k(f)=\frac{(q-1)q^{n-h-1}}{\Phi(T^{n-h})}\langle \mathcal{N}^U_{0,\Lambda_k}\rangle=\langle \mathcal{N}^U_{0,\Lambda_k}\rangle. \]

Combining \eqref{eq:Psi} with Lemma \ref{lem:N-Psi}, and subtracting the mean, we obtain,
\begin{equation}\label{eq:difference}
\mathcal{N}^U_{0,\Lambda_k}(T^{h+1}B;h)-\langle \mathcal{N}^U_{0,\Lambda_k}\rangle=\frac{1}{q^{n-h-1}} \sum_{\substack{\chi\not = \chi_0, \pmod{T^{n-h}}\\\text{ even}}}\overline{\chi}(B^*)M_0(n;\Lambda_k\chi).
\end{equation}

Applying \eqref{eq:M} to equation \eqref{eq:difference}, 
\[\mathrm{Var}(\mathcal{N}^U_{0,\Lambda_k})=\frac{1}{q^{n-h-1}} \sum_{\substack{B^*\pmod{T^{n-h}}\\B^*(0)=1}} \frac{1}{q^{2(n-h-1)}} \left| \sum_{\substack{\chi\not = \chi_0, \pmod{T^{n-h}}\\\text{ even}}}\overline{\chi}(B^*)M_0(n;\Lambda_k\chi)\right|^2.\]

Expanding, and using orthogonality relations, we get 
\begin{align*}
\mathrm{Var}(\mathcal{N}^U_{0,\Lambda_k})=&\frac{1}{q^{2(n-h-1)}} \sum_{\substack{\chi_1,\chi_2\not = \chi_0, \pmod{T^{n-h}}\\\text{ even}}}M_0(n;\Lambda_k\chi_1)\overline{M_0(n;\Lambda_k\chi_2)}\\&\times \frac{1}{q^{n-h-1}} \sum_{\substack{B^*\pmod{T^{n-h}}\\B^*(0)=1}} \overline{\chi_1}(B^*) \chi_2(B^*)\\
=&\frac{1}{q^{2(n-h-1)}} \sum_{\substack{\chi\not = \chi_0, \pmod{T^{n-h}}\\\text{ even}}}\left| M_0(n;\Lambda_k\chi)\right|^2.
\end{align*}
Therefore, by Lemma \ref{lem:M0-vM-interval}, we get
\[\mathrm{Var}(\mathcal{N}^U_{0,\Lambda_k})=\frac{q^{h+1}}{q^{n-h-1}} \sum_{\substack{\chi\not = \chi_0, \pmod{T^{n-h}}\\\text{even}}}\left|\sum_{
 \substack{j_1+\cdots+j_k=n\\1\leq j_1,\dots,j_k }} 
 \mathrm{Tr}(\Theta_\chi^{j_1})\cdots \mathrm{Tr}(\Theta_\chi^{j_k}) \right|^2\left(1+O\left(\frac{1}{\sqrt{q}}\right)\right).\]

We have now all the elements to prove Theorem \ref{thm:N0-vM}.
\begin{thm} Let $k\leq n$. 
 As $q\rightarrow \infty$,  
 \[\mathrm{Var}(\mathcal{N}^U_{0,\Lambda_k})\sim q^{h+1}\int_{\mathrm{U}(n-h-2)}\left|\sum_{
 \substack{j_1+\cdots+j_k=n\\1\leq j_1,\dots,j_k }} 
 \mathrm{Tr}(U^{j_1})\cdots \mathrm{Tr}(U^{j_k}) \right|^2dU.\]
\end{thm}

\begin{proof}
 In order to apply the equidistribution result of Katz \cite{Katz-question-KR}, we need to take the sum over the primitive characters. There are $q^{n-h-1}\left(1-\frac{1}{q}\right)$ of those. For the non-primitive characters, we bound $|M(n;\Lambda_k\chi)|\ll q^{n/2}$, and there are  $O(q^{n-h-2})$  of these. This gives the desired result. 
\end{proof}

  \subsection{Average of the von Mangoldt convolution function over arithmetic progressions}

We are interested in studying the following sum, which is the von Mangoldt convolution analogue to the problem in Section 5 of  
 \cite{Keating-Rudnick}:
\[\mathcal{S}^U_{\Lambda_k,n,Q}(A)=\sum_{\substack{f\in \mathcal{M}_n\\f\equiv A\pmod{Q}}}\Lambda_k(f),\]
where $Q$ is square-free and $A$ is coprime to $Q$. 
First we consider 
\[M(n;\Lambda_k\chi)=\sum_{\substack{f\in \mathcal{M}_n}}\Lambda_k(f)\chi(f).\]
 Lemma \ref{lem:genML} becomes the following statement. 
 \begin{lem} \label{lem:M-vM-U} Let $\chi$ be odd.  For $k\leq n$, we have
\[
 M(n;\Lambda_k\chi)= (-1)^k q^{n/2}\sum_{\substack{j_1+\cdots+j_k=n\\1\leq j_1,\dots, j_k}} \mathrm{Tr}(\Theta_{\chi}^{j_1})\cdots \mathrm{Tr}(\Theta_{\chi}^{j_k}),
 \]
 and $M(n;\Lambda_k\chi)=0$ for $1\leq n < k$. 
 \end{lem}

Following a similar argument to \cite[Section 4.1]{Keating-Rudnick-moebius}, which uses the orthogonality relations of Dirichlet characters to detect the arithmetic progression, 
\begin{equation}\label{eq:arith-prog}
\mathcal{S}^U_{\Lambda_k,n,Q}(A)=\frac{1}{\Phi(Q)}\sum_{\substack{f\in \mathcal{M}_n\\(f,Q)=1}}\Lambda_k(f)+
\frac{1}{\Phi(Q)}\sum_{\chi\not = \chi_0} \overline{\chi(A)} M(n;\Lambda_k\chi).
\end{equation}
\begin{lem} We have that 
\[\mathcal{S}^U_{\Lambda_k,n,Q}(A)=
\frac{1}{\Phi(Q)}\sum_{\substack{f\in \mathcal{M}_n\\(f,Q)=1}}\Lambda_k(f)\left(1+O\left(\frac{1}{q}\right)\right)=\frac{q^n}{\Phi(Q)}\binom{n-1}{k-1}  +O(q^{n-1}).\]
\end{lem}
\begin{proof}
The proof is identical to that of Lemma \ref{lem:vM-S}.
\end{proof}

As usual, we are interested in the variance. 
\begin{align*}
 \mathrm{Var}(\mathcal{S}^U_{\Lambda_k,n,Q})=& \frac{1}{\Phi(Q)}\sum_{\substack{A\pmod{Q}\\(A,Q)=1}}  \left|\mathcal{S}^U_{\Lambda_k,n,Q}(A)- \frac{1}{\Phi(Q)}\sum_{\substack{f\in \mathcal{M}_n\\(f,Q)=1}}\Lambda_k(f)\right|^2.
\end{align*}

Applying this to equation \eqref{eq:arith-prog}, we get 
\begin{equation}\label{eq:Var-U-S}
\mathrm{Var}(\mathcal{S}^U_{\Lambda_k,n,Q})=\frac{1}{\Phi(Q)^2}\sum_{\chi\not = \chi_0}|M(n;\Lambda_k\chi)|^2.
\end{equation}
From the bound of $M(n;\Lambda_k\chi)$ provided by the Riemann Hypothesis,  we obtain that the contribution from the even characters is 
\[\ll \frac{\Phi_{\text{ev}}(Q)}{\Phi(Q)^2}q^n \ll \frac{q^{n-1}}{\Phi(Q)},\]
where $\Phi_{\text{ev}}(Q)=\frac{\Phi(Q)}{q-1}$ is the number of even characters. This allows us to rewrite \eqref{eq:Var-U-S} as  
\[\mathrm{Var}(\mathcal{S}^U_{\Lambda_k,n,Q})=\frac{1}{\Phi(Q)^2}\sum_{\chi \text{ odd and primitive } }|M(n;\Lambda_k\chi)|^2 + O\left( \frac{q^{n-1}}{\Phi(Q)}\right).\]
By Lemma \ref{lem:M-vM-U}, we get 
\[\mathrm{Var}(\mathcal{S}^U_{\Lambda_k,n,Q}) =\frac{q^n}{\Phi(Q)^2}\sum_{\chi \text{ odd and primitive } }\left|\sum_{\substack{j_1+\cdots+j_k=n\\1\leq j_1,\dots, j_k}} \mathrm{Tr}(\Theta_{\chi}^{j_1})\cdots \mathrm{Tr}(\Theta_{\chi}^{j_k}),
 \right|^2 + O\left( \frac{q^{n-1}}{\Phi(Q)}\right).\]

By applying Katz's equidistribution Theorem from \cite{Katz-question-KR-Dirichlet}, we obtain the statement of Theorem \ref{thm:vM-U-arith}. 
\begin{thm} For $k\leq n$, as $q\rightarrow \infty$,
\[\mathrm{Var}(\mathcal{S}^U_{\Lambda_k,n,Q}) \sim \frac{q^n}{|Q|} \int_{\mathrm{U}(\deg(Q)-1)} \left|\sum_{\substack{j_1+\cdots+j_k=n\\1\leq j_1,\dots, j_k}} \mathrm{Tr}(U^{j_1})\cdots \mathrm{Tr}(U^{j_k}) \right|^2dU.\]
\end{thm}
\subsection{Relationship with Random Matrix Theory} 

In this section we discuss what is known about the computation of the integral 
\begin{equation}I_{\Lambda_k}^U(n;N):=\label{eq:int-unitary}\int_{\mathrm{U}(N)} \left|\sum_{\substack{j_1+\cdots+j_k=n\\1\leq j_1,\dots, j_k}} \mathrm{Tr}(U^{j_1})\cdots \mathrm{Tr}(U^{j_k}) \right|^2dU.
\end{equation}

In \cite{Diaconis-Shahshahani} Diaconis and Shahshahani prove the following result (see also \cite[Theorem 2.1] {Diaconis-Evans} and \cite{Diaconis-Gamburd}).
\begin{thm} \label{thm:DS}
\cite[Theorem 2]{Diaconis-Shahshahani} Let $U$ be Haar distributed on $\mathrm{U}(N)$. Let ${\bf a}=(a_1,\dots, a_k)$,  ${\bf b}=(b_1,\dots, b_k)$, with $a_i, b_i \in \Z_{\geq0}$. Then for $N\geq \max\left\{\sum_{j=1}^ka_jj, \sum_{j=1}^kb_jj\right\}$, 
\[\mathbb{E}\left(\prod_{j=1}^k (\mathrm{Tr}(U^j))^{a_j}  \overline{(\mathrm{Tr} (U^j))}^{b_j}  \right)=\delta_{\bf a b} \prod_{j=1}^n j^{a_j} a_j!\]

Also, from \cite[Theorem 2.1] {Diaconis-Evans}, we have
\[\mathbb{E}\left(\mathrm{Tr} (U^{j_1})\overline{\mathrm{Tr} (U^{j_2})}\right)=\delta_{j_1,j_2} \max \{j_1,N\}.\]
\end{thm}
The cases $k=1,2$ of the above result were also known to Dyson \cite{Dyson}.

In order to compute the integral from \eqref{eq:int-unitary}, we are interested in the case when ${\bf a}={\bf b}$, so that the formula is nonzero. In the integral setting, Theorem \ref{thm:DS} implies that when $n \le N$,  
\begin{equation}\label{eq:int-unit}
\int_{\mathrm{U}(N)} \left|\sum_{\substack{j_1+\cdots+j_k=n\\1\leq j_1,\dots, j_k}} \mathrm{Tr}(U^{j_1})\cdots \mathrm{Tr}(U^{j_k}) \right|^2dU= \sum_{\substack{a_1+\cdots+a_n=k\\a_1+a_22+\cdots +a_nn=n}} \prod_{j=1}^n j^{a_j} a_j!
\end{equation}
The product inside the sum is the number of permutations in $\mathbb{S}_n$ that commute with a permutation that has $a_1$ cycles of length $1$, $a_2$ cycles of length $2$, etc. In other words, we are summing the orders of all the possible centralizers of permutations with exactly $k$ cycles. 
 
The above results are restricted to $n\leq N$ or to $k=1$. If we want to consider other possibilities we must work with a more general setting.  Consider for a matrix $U\in \mathrm{U}(N)$
\[\mathcal{L}_U(s):=\det(I-sU)=\prod_{n=1}^N (1-se^{-i\theta_n}).\]
This $L$-function satisfies the functional equation 
\[\mathcal{L}_U(s)=(-1)^N\det U^*s^N \mathcal{L}_{U^*} (1/s).\]

Conrey and Snaith \cite{Conrey-Snaith-correlations} prove the following result.  
\begin{thm}\label{eq:rmt-trace-unitary}\cite[Theorem 3]{Conrey-Snaith-correlations} Let $A=\{\alpha_1,\dots,\alpha_k\}, B=\{\beta_1,\dots,\beta_k\}$, such that $\re(\alpha_j), \re(\beta_j)>0$, then 
\begin{equation}\label{eq:unitary-generating}\int_{\mathrm{U}(N)} \prod_{\alpha\in A} (-e^{-\alpha})\frac{\mathcal{L}_U'}{\mathcal{L}_U} (e^{-\alpha}) \prod_{\beta\in B} (-e^{-\beta})\frac{\mathcal{L}_{U^*}'}{\mathcal{L}_{U^*}} (e^{-\beta}) dU
\end{equation}
 is equal to 
 \[\sum_{\substack{S\subseteq A, T\subseteq B\\ |S|=|T|}} e^{-N(\sum_{\hat{\alpha}\in S}\hat{\alpha}+\sum_{\hat{\beta}\in T}\hat{\beta})}\frac{Z(S,T)Z(S^-,T^-)}{Z^\dagger (S,S^-)Z^\dagger (T,T^-)} \sum_{\substack{(A-S)+(B-T)=W_1+\cdots+W_R\\|W_r|\leq 2}}
\prod_{r=1}^R H_{S,T}(W_r).
 \]
Here 
\[S^-=\{-\hat{\alpha}, \hat{\alpha}\in S\}, T^-=\{-\hat{\beta}, \hat{\beta}\in S\},\]
and 
\[Z(A,B)=\prod_{\substack{\alpha \in A\\\beta\in B}} z(\alpha+\beta),\]
where $z(\alpha)=\frac{1}{1-e^{-\alpha}}$, and the dagger imposes the additional restriction that a factor $z(\alpha)$ is omitted when the argument is zero.

Finally, the sum over $W_r$ is a sum over all the different set partitions of $A-S$ and $B-T$ and   
\[H_{S,T}(W)=\begin{cases}\sum_{\hat{\alpha}\in S} \frac{z'}{z}(\alpha-\hat{\alpha}) -\sum_{\hat{\beta}\in T} \frac{z'}{z}(\alpha+\hat{\beta}) & W=\{\alpha\}\subseteq A-S,\\
\sum_{\hat{\beta}\in T} \frac{z'}{z}(\beta-\hat{\beta}) -\sum_{\hat{\alpha}\in S} \frac{z'}{z}(\beta+\hat{\alpha}) & W=\{\beta\}\subseteq B-T,\\
\left(\frac{z'}{z}\right)'(\alpha+\beta) & W=\{\alpha, \beta\}, \alpha \in A-S, \beta\in B-T,\\
1 & W=\varnothing.
              \end{cases}\]
\end{thm}
Recall that, as in the symplectic case of Theorem \ref{thm:MS},  $\frac{z'}{z}(\alpha)=-\frac{e^{-\alpha}}{1-e^{-\alpha}}$, $\left(\frac{z'}{z}\right)'(\alpha)=\frac{e^{-\alpha}}{(1-e^{-\alpha})^2}$.

 From now on, we will set $x_i:=e^{-\alpha_i}$ and  $y_i:=e^{-\beta_i}$. 
We will proceed to investigate some particular swaps of the above result. Recall that our goal is to evaluate integral \eqref{eq:int-unitary}. Notice that Theorem \ref{eq:rmt-trace-unitary} provides a generating function for $\mathbb{E}\left(\prod_{j=1}^k (\mathrm{Tr}(M^j))^{a_j}  \overline{(\mathrm{Tr} (M^j))}^{b_j}  \right)$  and therefore, we are interested in the case in which the exponent of $x_i$ is matched with the exponent of $y_i$, with a fixed total weight. 
 
\subsubsection{$0$-swap terms} The $0$-swap terms arise from taking $S$ and $T$ empty. We will recover equation \eqref{eq:int-unit}  in this way. The $0$-swap terms of integral \eqref{eq:unitary-generating} are then given by terms of the form
\[\sum_{\substack{A+B=W_1+\cdots+W_R\\|W_r|\leq 2}}
\prod_{r=1}^R H_{\varnothing,\varnothing}(W_r).\]
Notice that these terms contribute when $A+B=\sum_{j=1}^k \{\alpha_j, \beta_{\sigma(j)}\}$. In this way we find the terms
\begin{align*}
\sum_{\sigma\in \mathbb{S}_k} \prod_{j=1}^k \frac{x_j y_{\sigma(j)}}{(1-x_jy_{\sigma(j)})^2}=&
\sum_{\sigma\in \mathbb{S}_k} \prod_{j=1}^k \sum_{\ell=1}^\infty \ell x_j^\ell y_{\sigma(j)}^\ell\\
=&\sum_{\sigma\in \mathbb{S}_k}\sum_{\ell_1,\dots,\ell_k=1}^\infty \ell_1 x_1^{\ell_1} y_{\sigma(1)}^{\ell_1}\cdots \ell_k x_k^{\ell_k} y_{\sigma(k)}^{\ell_k}.
\end{align*}
The above can be more clearly written as 
\[\sum_{\ell_1,\dots,\ell_k=1}^\infty \ell_1\cdots  \ell_k x_1^{\ell_1} \cdots x_k^{\ell_k}\sum_{\sigma\in \mathbb{S}_k}y_{\sigma(1)}^{\ell_1}\cdots  y_{\sigma(k)}^{\ell_k}.\]
We regroup all the $x_i$'s according to their common exponent in the above sum. More precisely, say that for exponent $j$, we have $x_{i_{j,1}}\cdots x_{i_{j,a_j}}$ and those are the only $x_i$'s with exponent $j$. The terms that contribute to integral \eqref{eq:int-unitary} are those that are multiplied by $y_{i_{j,1}}\cdots y_{i_{j,a_j}}$. But this means that we must have \[y_{i_{j,1}}\cdots y_{i_{j,a_j}}=y_{\sigma(i_{j,1})}\cdots y_{\sigma(i_{j,a_j})}.\] 
In sum, the $\sigma\in \mathbb{S}_k$ that contribute to the final sum giving the value of integral \eqref{eq:int-unitary} are exactly those that satisfy $\sigma \in \mathbb{S}_{\{i_{j,1},\cdots i_{j,a_j}\}}$. There are $a_j!$ of those. To get  \eqref{eq:int-unitary}, we must isolate the coefficient corresponding to the restriction $a_1+a_22+\cdots +a_nn=n$. This recovers formula \eqref{eq:int-unit}. We remark that the  0-swaps terms contribute to \eqref{eq:int-unitary} not only when $n\leq N$, but also when $n>N$. In this latter case, the final value has also contributions from other terms. 

\subsubsection{$1$-swap terms} The 1-swap terms are more difficult to compute than the 0-swap terms. In this section we do this computation for $k=2$, when $n \le 2N+1$.

We first consider $S = \{\alpha_1\}$, $T = \{\beta_1\}$. The term corresponding to these $S$ and $T$ is given by
\[-x_1^Ny_1^N \frac{x_1y_1}{(1-x_1y_1)^2} \left(\frac{x_2y_2}{(1-x_2y_2)^2} + \frac{x_2y_2(1-x_1y_1)^2}{(x_2-x_1)(1-x_2y_1)(y_2-y_1)(1-x_1y_2)}\right).\]
The remaining three terms are similar; adding all four together, with a little computation, yields
\begin{align*}
&-\frac{(x_1^Ny_1^N+x_2^Ny_2^N)x_1y_1x_2y_2}{(1-x_1y_1)^2(1-x_2y_2)^2} -\frac{(x_1^Ny_2^N+x_2^Ny_1^N)x_1y_1x_2y_2}{(1-x_1y_2)^2(1-x_2y_1)^2}\\
&-\frac{x_1y_1x_2y_2}{(x_2-x_1)(y_2-y_1)}\left(\frac{x_1^Ny_1^N+x_2^Ny_2^N}{(1-x_2y_1)(1-x_1y_2)}-\frac{x_1^Ny_2^N+x_2^Ny_1^N}{(1-x_1y_1)(1-x_2y_2)}\right).
\end{align*}
The denominator $(x_2-x_1)(y_2-y_1)$ divides the numerator in parentheses, and we can expand the remaining denominators into power series, yielding
\begin{align}\nonumber
=&-(x_1^Ny_1^N+x_2^Ny_2^N)\sum_{j,h\geq 1} jh x_1^jy_1^jx_2^h y_2^h -(x_1^Ny_2^N+x_2^Ny_1^N)\sum_{j,h\geq 1} jh x_1^jy_2^jx_2^h y_1^h\nonumber\\
&-\left(\sum_{j,h\geq 1} x_1^jy_1^jx_2^h y_2^h\right)\left(\sum_{j,h\geq 1} x_1^jy_2^jx_2^h y_1^h\right)\nonumber\\
&\times \left( -\sum_{j=0}^N x_1^jx_2^{N-j} \sum_{h=0}^N y_1^hy_2^{N-h}+\sum_{j=0}^{N-1} x_1^jx_2^{N-1-j} \sum_{h=0}^{N-1} y_1^hy_2^{N-1-h} + \sum_{j=0}^{N-1} x_1^{j+1}x_2^{N-j} \sum_{h=0}^{N-1} y_1^{h+1}y_2^{N-h} \right .\label{suminparenthesis}\\
& \left. -\sum_{j=0}^{N-2} x_1^{j+1}x_2^{N-1-j} \sum_{h=0}^{N-2} y_1^{h+1}y_2^{N-1-h}\right).\nonumber
\end{align}

Our goal now is to find the sum of all coefficients of terms of the form $x_1^ax_2^by_1^ay_2^b$, where $a + b = n$. Our strategy for each term in the sum above is to isolate terms where the exponent of $x_1$ is equal to that of $y_1$ and the same is true for $x_2$ and $y_2$ before finding the coefficient sum. We first address the first two sums above, which are similar to the 0-swap terms:
\[-(x_1^Ny_1^N+x_2^Ny_2^N)\sum_{j,h\geq 1} jh x_1^jy_1^jx_2^h y_2^h -(x_1^Ny_2^N+x_2^Ny_1^N)\sum_{j,h\geq 1} jh x_1^jy_2^jx_2^h y_1^h\]
For these sums, in line with formula \eqref{eq:int-unit}, we get
\begin{align*}
&-\frac{(n-N-1)(n-N)(n-N+1)}{3}-\delta_{n\equiv 0 \pmod{2}}\delta_{n\geq 2N+2}\frac{(n-N)^2-N^2}{2}.
\end{align*}
Since $n\leq 2N+1$ by assumption, we will ignore the second term.

We now turn our attention to the sums in parentheses \eqref{suminparenthesis}. We first consider the first double sum:
\[\left(\sum_{j_0,h_0\geq 1} x_1^{j_0}y_1^{j_0}x_2^{h_0} y_2^{h_0}\right)\left(\sum_{j_1,h_1\geq 1} x_1^{j_1}y_2^{j_1}x_2^{h_1} y_1^{h_1}\right)\left(\sum_{j_2=0}^N x_1^{j_2}x_2^{N-j_2} \sum_{h_2=0}^N y_1^{h_2}y_2^{N-h_2}\right).\]
As above, our goal is to isolate terms where the exponent of $x_1$ is equal to that of $y_1$ and the same for $x_2$ and $y_2$. These are precisely the terms where
\[j_1+j_2=h_1+h_2.\]
We now take the limit as $x_1 \to x_2$ and $y_1 \to y_2$; we will call these $x$ and $y$ going forward. We need to determine the coefficient of the term $x^ny^n$; we call this coefficient $f(n,N)$. Note that
\[j_0+h_0+j_1+h_1+j_2+N-j_2=n,\]
or equivalently
\begin{equation}\label{eq:exponentsum}
j_0+h_0+j_1+h_1=n-N.
\end{equation}
For each tuple $j_0,h_0,j_1,h_1$ satisfying \eqref{eq:exponentsum}, we first count the number of options of $j_2,h_2$ satisfying $j_1-h_1=h_2-j_2$. 

Write $k:=j_1-h_1$ (where $k$ can be negative, and $|k|\leq N$ since otherwise there are no choices of $j_2$ and $h_2$ with $k = j_2-h_2$). For a fixed $k$, there are $N-|k|+1$ ways to choose $j_2$ and $h_2$ with $j_2-h_2 = k$. It remains to count solutions to
\begin{equation}\label{eq:sumk}
j_0+h_0+2h_1=n-N-k,
\end{equation}
keeping in mind that $j_0,h_0,h_1 \ge 1$. The number of solutions to \eqref{eq:sumk} is 
\[\begin{cases}
   \left(\frac{n-N-k-2}{2}\right)^2 & n-N-k-4\geq 0 \mbox{ even},\\
   \left(\frac{n-N-k-3}{2}\right)  \left(\frac{n-N-k-1}{2}\right)& n-N-k-4\geq 0 \mbox{ odd}.\\
  \end{cases}\]
Thus the final coefficient of $x^ny^n$ is 
\begin{align*}
f(n,N) = \sum_{\substack{-N\leq k \leq N\\k\leq n-N-4}} &(N-|k|+1)\left(\frac{n-N-k-2}{2}\right)^2 \\
&-\frac{1}{4}\sum_{\substack{-N\leq k \leq N\\k\leq n-N-4}} (N-|k|+1) \left(\frac{1+(-1)^{n-N-k+1}}{2}\right).
\end{align*}
Since $n\leq 2N+1$, the restriction that $k\leq n-N-4$ is more stringent than the restriction that $k\leq N$. 
We also need to assume that $n\geq N+2$, since otherwise the contribution is empty.  

Thus after expanding and simplifying, the coefficient of $x^ny^n$ is given by
\begin{align*}
f(n,N) =&\frac{-n^4+4(2N+3)n^3-2(6N^2+24N+19)n^2+(8N^3+48N^2+80N+36)n}{48}\\&-
\frac{4N^4 + 32N^3 + 80N^2 + 64N + 9}{96}\\
&+\frac{(-1)^n+2(-1)^{n+N}}{32}.
\end{align*}

The remaining sums contribute coefficients $-f(n,N-1)$, $-f(n-2,N-1)$, and $f(n-2,N-2)$. Returning our attention to the full sum, and combining with our computation above, the 1-swap terms are then given by
\begin{align*}
&-\frac{(n-N-1)(n-N)(n-N+1)}{3} \\
&+ f(n,N) - f(n,N-1) - f(n-2,N-1) + f(n-2,N-2)\\
\end{align*}

Plugging in our expression for $f(n,N)$ gives
\begin{align*}
=& \frac{-4(n-N)^3+6(n-N)^2-20(n-N)+21+3(-1)^{N+n}}{12}
\end{align*}
which is the final result for the 1-swap terms.

For the final result for $N+2\leq n\leq 2N+1$ we would need to consider the sum with the 0-swap terms, which gives
\[\frac{n(4n^2+3n-4)}{24}+\frac{(-1)^n n^2}{8}+\frac{-4(n-N)^3+6(n-N)^2-20(n-N)+21+3(-1)^{N+n}}{12}\]

In the case when $k = 2$, we have shown the following corollary.
\begin{cor}
Let $k = 2$. The integral $I_{\Lambda_2}^U(n,N)$ is given by
\begin{align*}
\int_{U(n)} &\left|\sum_{\substack{j_1 + j_2 = n \\ 1 \le j_1,j_2}} \mathrm{Tr}(U^{j_1})\mathrm{Tr}(U^{j_2}) \right|^2 dU \\ &=\begin{cases}
            \frac{n(4n^2+3n-4)}{24}+\frac{(-1)^n n^2}{8} &2\leq n\leq N+1,\\
            \frac{n(4n^2+3n-4)}{24}+\frac{(-1)^n n^2}{8}+\frac{-4(n-N)^3+6(n-N)^2-20(n-N)+21+3(-1)^{N+n}}{12} 
            & N+2\leq n\leq 2N+3.
           \end{cases}
\end{align*}
In the limit as $N \to \infty$, $I_{\Lambda_2}^U(n,N)$ is asymptotic to
\[I_{\Lambda_2}^U(n,N)\sim \gamma_{\Lambda_2}^U(c)N^3,\]
where $c = n/N$ and
\[\gamma_{\Lambda_2}^U(c)=\begin{cases} 
                \frac{c^3}{6} & 0\leq c \leq 1,\\ \\
                \frac{c^3}{6}+\frac{(1-c)^3}{3} &  1\leq c \leq 2.\\
               \end{cases}
\] 
\end{cor}

\bibliographystyle{amsalpha}

\bibliography{Bibliography}

\end{document}